\documentclass[reqno]{amsart}

\usepackage[margin=1.4in]{geometry}
\usepackage{amssymb, amsmath, amsthm, amsfonts, amscd}
\usepackage[shortlabels]{enumitem}
\usepackage{tikz, tikz-cd}
\usepackage[all]{xy}
\usepackage[new]{old-arrows}
\usepackage{pst-node}
\usepackage{varwidth}
\usepackage{graphicx}
\usepackage{subfig}
\usepackage{tabularx, booktabs}
\usepackage{multirow, multicol}
\usepackage{xcolor}
\usepackage{mathrsfs}
\usepackage{yfonts, euscript}
\usepackage{csquotes, dirtytalk}
\usepackage[pagewise]{lineno}
\usepackage[colorlinks=true, allcolors=blue]{hyperref}

\usetikzlibrary{arrows, arrows.meta}

\newtheorem{thm}{{\bf Theorem}}[section]
\newtheorem{lemma}[thm]{{\bf Lemma}}
\newtheorem{prop}[thm]{{\bf Proposition}}
\newtheorem{cor}[thm]{{\bf Corollary}}
\newtheorem*{defn}{{\bf Definition}}

\newtheorem*{rmk}{{\bf Remark}}
\newtheorem*{conv}{{\bf Convention}}

\newcommand{\RNum}[1]{\uppercase\expandafter{\romannumeral #1\relax}}

\begin{document}


\title[Triangular and Unitriangular Factorization]{Triangular and Unitriangular Factorization of \\ Twisted Chevalley Groups}


\author{Shripad M. Garge}
\address{Department of Mathematics, 
	Indian Institute of Technology Bombay,\newline \indent
	Powai, Mumbai 400076, India}
\email{\href{mailto:shripad@math.iitb.ac.in}{shripad@math.iitb.ac.in}, \href{mailto:smgarge@gmail.com}{smgarge@gmail.com}}

\author{Deep H. Makadiya}
\address{Department of Mathematics, 
	Indian Institute of Technology Bombay,\newline \indent
	Powai, Mumbai 400076, India}
\email{\href{mailto:deepmakadia25.dm@gmail.com}{deepmakadia25.dm@gmail.com}}




\begin{abstract}
    The existence of triangular and unitriangular factorizations has been extensively studied for untwisted Chevalley groups, as well as for twisted Chevalley groups of types other than ${}^2A_{2n} \ (n \geq 1)$ (see~\cite{SSV1,SSV2}). 
    However, the case of twisted Chevalley groups of type ${}^2A_{2n} \ (n \geq 1)$, has remained unresolved in the general setting of commutative rings. 
    Prior work by A.~Smolensky~\cite{AS} addressed this case only over certain fields, including finite fields and the field of complex numbers. 
    These results indicate that, even over fields, the ${}^2A_{2n}$ case demands more refined techniques, reflecting the difficulty of extending such factorizations to the broader class of commutative rings. 

    In this paper, we introduce two new classes of commutative rings: those satisfying the \emph{special stable range one condition} and those that are \emph{$\theta$-complete}. We discuss their basic properties and provide illustrative examples. Our main result establishes the existence of triangular and unitriangular factorizations for twisted Chevalley groups of type ${}^2A_{2n}$ over a certain class of commutative rings, which includes all fields, all local rings (with mild restrictions), and several other important classes of rings.
\end{abstract}


\maketitle 


\section{Introduction}\label{sec:intro}

Let $\Phi$ be a reduced irreducible root system, and let $R$ be a commutative ring with unity. Let $G_{sc}(\Phi, R)$ denote the simply connected Chevalley group of type $\Phi$, and let $E_{sc}(\Phi, R)$ be its elementary subgroup.
Since we work exclusively with simply connected groups, we will omit the subscript $\mathrm{sc}$ and use the simplified notation $G(\Phi, R)$ and $E(\Phi, R)$ in place of $G_{sc}(\Phi, R)$ and $E_{sc}(\Phi, R)$, respectively.


The problem of finding ``$LU$-type" decomposition of Chevalley groups $G = G(\Phi, R)$ (or $E(\Phi, R)$ or any classical groups) over rings has been a subject of sustained interest over the last few decades. Two prominent forms of such decompositions are as follows:

\begin{itemize}
    \item \textbf{Triangular factorizations}, in which the group $G$ is expressed in the form
    \[
        G = B U^{-} U^{+} U^{-} \cdots U^{\pm} = T U^{+} U^{-} \cdots U^{\pm},
    \]
    where $T = T(\Phi, R)$ is a (fixed) split maximal torus, $B = B(\Phi, R)$ is a standard Borel subgroup containing $T$, and $U^{+} = U(\Phi, R)$ and $U^{-} = U^{-}(\Phi, R)$ are the unipotent radicals of $B$ and its opposite $B^{-}$, respectively.

    \item \textbf{Unitriangular factorizations}, which aim to express \( G \) in the form:
    \[
        G = U^{+} U^{-} U^{+} \cdots U^{\pm},
    \]
    with \( U^{+} \) and \( U^{-} \) as above.
\end{itemize}

A central question in this context is whether such factorizations exist for a given ring $R$. When they do, a natural objective is to determine the minimal number of unipotent factors required to express the group.

In general, the existence of such factorizations cannot be guaranteed for arbitrary rings. As a result, one typically restricts attention to specific classes of rings for which triangular and unitriangular factorizations are expected to hold. 

These questions have been the subject of extensive study, and a substantial body of literature is dedicated to their investigation. For a detailed historical overview, including developments related to classical groups, we refer the reader to~\cite{SSV1,SSV2} and the references therein.


\subsection{Some Known Results} 

In this subsection, we recall the main results from \cite{SSV1} and \cite{SSV2}. Before doing so, we introduce the stable range one condition for rings.

A ring $R$ satisfies the stable range one condition \textbf{$(SR_1)$} (in the sense of Bass~\cite{HB}), if for all $a,b \in R$ such that $aR + bR = R$, there exists a $z \in R$ for which $a + bz \in R^*$.

\begin{thm}[{\cite[Theorem 1.1]{SSV2}}]\label{thm:SSV2}
    \normalfont
    Let $\Phi$ be a reduced irreducible root system and $R$ be a commutative ring satisfying stable range one $(SR_1)$ condition. Then the elementary Chevalley group $E(\Phi, R)$ admits triangular decomposition
    \[
        E(\Phi, R) = H(\Phi, R) \, U(\Phi, R) \, U^{-}(\Phi, R) \, U(\Phi, R),
    \]
    where $H(\Phi, R) = T(\Phi, R) \cap E(\Phi, R)$.
    Conversely, if the above triangular decomposition holds for some elementary Chevalley group, then $R$ satisfies the $(SR_1)$ condition.
\end{thm}

\begin{rmk}
        The triangular factorization described in the preceding theorem is also referred to as the \emph{Gauss decomposition}.
\end{rmk}

\begin{thm}[{\cite[Theorem 1]{SSV1}}]\label{thm:SSV1}
    \normalfont
    Let $\Phi$ be a reduced irreducible root system and $R$ be a commutative ring satisfying stable range one $(SR_1)$ condition. Then the elementary Chevalley group $E(\Phi, R)$ admits unitriangular factorization
    \[
        E(\Phi, R) = U(\Phi, R) \, U^{-}(\Phi, R) \, U(\Phi, R) \, U^{-}(\Phi, R),
    \]
    of length $4$.
\end{thm}

As noted in \cite{SSV1} and \cite{AS}, the theorems stated above also hold for twisted Chevalley groups of all types except ${}^2A_{2n}$, under somewhat stronger conditions on the ring. The remaining case of groups of type ${}^2A_{2n}$ demands a significantly more delicate analysis, which is the primary focus of this paper. Nevertheless, we formulate the theorems for twisted Chevalley groups of all types. Our main results are presented in the following subsection.


\subsection{Main Results} 

Let $R$ be a commutative ring with unity, and let $\Phi$ be an irreducible root system of type $A_n \ (n \geq 2)$, $D_n \ (n \geq 4)$ or $E_6$.


\subsubsection{}

Let $G = G(\Phi, R)$ be the (simply connected) Chevalley group over $R$ of type $\Phi$, and let $E = E(\Phi, R)$ denote the corresponding elementary subgroup. Fix a simple root system $\Delta$ of $\Phi$. An angle-preserving permutation $\rho$ of $\Delta$ induces an automorphism of $G$, known as a \emph{graph automorphism}, which we also denote by $\rho$. Fix a nontrivial graph automorphism $\rho$ of $\Phi$. Then the order of $\rho$ is either $2$ or $3$, with order $3$ occurring only in the case $\Phi \sim D_4$. Let $\Phi_\rho$ denote the associated twisted root system.

Similarly, a ring automorphism $\theta: R \to R$ induces a group automorphism of $G$, also denoted by $\theta$, which is referred to as the \emph{ring automorphism} of $G$. Suppose that $R$ admits a ring automorphism $\theta$ whose order coincides with that of the graph automorphism $\rho$. Define $\sigma = \theta \circ \rho$; then $\sigma$ is an automorphism of $G$ of order $o(\rho)$.

The subgroup $G_\sigma = G_\sigma(\Phi, R) = \{ g \in G \mid \sigma(g) = g \}$ is called the \emph{twisted Chevalley group} over $R$ of type ${}^n X$ (where $X$ denotes the type of $\Phi$, and the superscript $n$ indicates the order of the automorphism $\sigma$). The corresponding elementary subgroup $E'_\sigma = E'_\sigma(\Phi, R)$ is referred to as the \emph{elementary twisted Chevalley group} over $R$ of type ${}^n X$.

For any subset \( A \subset G \), we define \( A_\sigma := A \cap G_\sigma \). In particular, $E_\sigma = E_\sigma(\Phi, R) := E(\Phi, R) \cap G_\sigma(\Phi, R)$. In a similar manner, we define \( U_\sigma = U_\sigma(\Phi, R) \), \( U_\sigma^{-} = U_\sigma^{-}(\Phi, R) \), \( T_{\sigma} = T_\sigma(\Phi, R) \), \( H_\sigma = H_\sigma(\Phi, R) \), and so on. Finally, let $H'_\sigma = H'_\sigma (\Phi, R):= H(\Phi, R) \cap E'_{\sigma}(\Phi, R)$. If $E_\sigma = E'_\sigma$ then $H_\sigma = H'_\sigma$. (All of these notions will be introduced in detail in the following section.)


\subsubsection{}

Let $R^*$ denote the group of units in $R$. 
Consider the following subsets of $R$:
\[
R_\theta = \{ r \in R \mid r = \bar{r} \} \quad \text{and} \quad R_{\theta}^{-} = \{ r \in R \mid r = -\bar{r} \}.
\]
It is clear that $R_\theta$ is a subring of $R$, and hence $R$ naturally inherits the structure of an $R_\theta$-algebra. 
If $2 \in R^*$, then $R$ admits the direct sum decomposition
\[
R = R_\theta \oplus R_{\theta}^{-}.
\]
Furthermore, if $2 \in R^{*}$ and $R_{\theta}^{-} \cap R^* \neq \emptyset$, then the $R_\theta$-module $R$ has rank $2$.
In this case, if $x \in R_\theta \cap R^*$ and $y \in R_{\theta}^{-} \cap R^*$, then $\{x, y\}$ forms a basis of $R$ as an $R_\theta$-module.
Finally, define the sets
\[
\mathcal{A}(R) = \{ (t,u) \in R^2 \mid t\bar{t} = u + \bar{u} \} \quad \text{and} \quad \mathcal{A}(R)^* = \{ (t,u) \in \mathcal{A}(R) \mid u \in R^* \}. 
\]


\subsubsection{}

Our first objective is to establish a result analogous to Theorem~\ref{thm:SSV2}. To this end, we introduces a new condition analogous to the stable range one condition $(SR_1)$.

\begin{defn}
    \normalfont
    Let $R$ be a commutative ring equipped with an involution $\theta$.
    \begin{enumerate}[(a)]
        \item A vector $(a, b, c) \in R^3$ is called \emph{special unitary completable} if there exists a matrix $A \in SU(3, R)$ such that the first row of $A$ is  $(a, b, c)$.
        \item The ring $R$ satisfies the \emph{special stable range one condition} $(SSR_1)$, if for every special unitary completable vector $(a, b, c) \in R^3$, there exists a pair $(z_1, z_2) \in \mathcal{A}(R)$ such that $a + b z_1 + c z_2 \in R^*$.
    \end{enumerate}
\end{defn}

\begin{rmk}
    To the best of the author's knowledge, this variant of the stable range one condition has not been previously documented in the existing literature. Examples of rings satisfying the $(\mathrm{SSR}_1)$ condition include fields, local rings (with mild conditions), and others (see Section~\ref{sec:SSR1}).
\end{rmk}

\begin{thm}[Triangular Decomposition]\label{thm:triangular decomposition}
    \normalfont
    Let $\Phi_\rho$ be a twisted root system of type ${}^2 A_n \ (n \geq 2), {}^2 D_n \ (n \geq 4), {}^2E_6$ or ${}^3 D_4$.
    Let $R$ be a commutative ring with unity, and suppose there exists a ring automorphism $\theta: R \longrightarrow R$ of order $2$ if $\Phi_\rho \not\sim {}^3 D_4$, or of order $3$ if $\Phi_\rho \sim {}^3 D_4$. 
    Assume that:
    \begin{enumerate}[(a)]
        \item Both $R$ and $R_\theta$ satisfy the $(SR_1)$ condition if $\Phi_\rho \sim {}^2 A_{2n+1}$ ($n \geq 1$), ${}^2 D_n$ ($n \geq 4$), ${}^2 E_6$, or ${}^3 D_4$;
        \item $R$ satisfies the $(SSR_1)$ condition if $\Phi_\rho \sim {}^2 A_{2}$;
        \item $R$ satisfies both the $(SR_1)$ and $(SSR_1)$ conditions if $\Phi_\rho \sim {}^2 A_{2n}$ ($n \geq 2$).
    \end{enumerate}
    Then the elementary twisted Chevalley group $E'_\sigma(\Phi, R)$ admits the triangular decomposition:
    \begin{align*}
        E'_\sigma(\Phi, R) &= H'_\sigma(\Phi, R) \, U_\sigma(\Phi, R) \, U^{-}_\sigma(\Phi, R) \, U_\sigma(\Phi, R) \\
        &= H'_\sigma(\Phi, R) \, U^{-}_\sigma(\Phi, R) \, U_\sigma(\Phi, R) \, U^{-}_\sigma(\Phi, R).
    \end{align*}
\end{thm}


\subsubsection{}

Our next goal is to prove a result analogous to Theorem~\ref{thm:SSV1}. Before proceeding, we address the following question: under what conditions does the equality $T_\sigma(\Phi, R) = H'_\sigma(\Phi, R)$ hold for $\Phi \sim A_{2n}$ ($n \geq 1$)? To answer this, we introduce a new class of rings. 

Define the set
\[
\mathcal{B}_1(R) = \{ u \in R \mid \text{there exists } t \in R \text{ such that } (t,u) \in \mathcal{A}(R)^{*}\}.
\]
If $\mathcal{A}(R)^* \neq \emptyset$ (equivalently, $\mathcal{B}_1(R) \neq \emptyset$), then for any $k \in \mathbb{N}$, define
\[
\mathcal{B}_k(R) = \{ u \in R \mid \text{there exist } u_1, \dots, u_k \in \mathcal{B}_1(R) \text{ such that } u = u_1 \dots u_k \}.
\]
Otherwise, set $\mathcal{B}_k(R) = \emptyset$.  
Now, define the sets  
\[
\mathcal{C}_{\text{even}}(R) = \bigcup_{\ell \in \mathbb{N}} \mathcal{B}_{2\ell}(R)
\quad \text{and} \quad
\mathcal{C}_{\text{odd}}(R) = \bigcup_{\ell \in \mathbb{N}} \mathcal{B}_{2\ell - 1}(R).
\]

\begin{defn}
    \normalfont
    Let $R$ be a commutative ring with unity, and suppose there exists an automorphism $\theta: R \to R$ of order $2$. We say that $R$ is $\theta$-complete if $\mathcal{C}_{\text{even}}(R) = R^*$.
\end{defn}

\begin{rmk}
    Again we remark that to the best of the author's knowledge, this is a new class of ring. Examples of $\theta$-complete rings include fields, local rings (with mild conditions), and others (see Section~\ref{sec:theta-complete ring}).
\end{rmk}

Note that, if $\mathcal{A}(R)^* \neq \emptyset$ then
\[
    \mathcal{B}_1(R) \subset \mathcal{B}_3 (R) \subset \mathcal{B}_5 (R) \subset \cdots \quad \text{and} \quad \mathcal{B}_2(R) \subset \mathcal{B}_4 (R) \subset \mathcal{B}_6 (R) \subset \cdots.
\]

\begin{defn}
    \normalfont
    If there exists a positive integer $k \in \mathbb{N}$ such that $\mathcal{C}_{\text{even}}(R) = \mathcal{B}_{2k}(R)$, then we say that $R$ has \emph{finite} $\mathcal{C}$-length. The minimal such $k$ for which this equality holds is called the \emph{$\mathcal{C}$-length} of $R$.
\end{defn}

\begin{thm}[Unitriangular Factorization]\label{thm:unitriangular decomposition}
    \normalfont
    Let $\Phi_\rho$ be a twisted root system of type ${}^2 A_n \ (n \geq 2), {}^2 D_n \ (n \geq 4), {}^2E_6$ or ${}^3 D_4$.
    Let $R$ be a commutative ring with unity, and suppose there exists a ring automorphism $\theta: R \longrightarrow R$ of order $2$ if $\Phi_\rho \not\sim {}^3 D_4$, or of order $3$ if $\Phi_\rho \sim {}^3 D_4$. 
    Assume that:
    \begin{enumerate}[(a)]
        \item Both $R$ and $R_\theta$ satisfy the $(SR_1)$ condition if $\Phi_\rho \sim {}^2 A_{2n+1}$ ($n \geq 1$), ${}^2 D_n$ ($n \geq 4$), ${}^2 E_6$, or ${}^3 D_4$;
        \item $R$ is $\theta$-complete with $\mathcal{C}$-length $k$ and satisfies the $(SSR_1)$ condition if $\Phi_\rho \sim {}^2 A_{2}$;
        \item $R$ is $\theta$-complete with $\mathcal{C}$-length $k$ and satisfies both the $(SR_1)$ and $(SSR_1)$ conditions if $\Phi_\rho \sim {}^2 A_{2n}$ ($n \geq 2$).
    \end{enumerate}
    Then the group $E'_\sigma(\Phi, R)$ admits the unitriangular factorization:
    \[
        E'_\sigma(\Phi, R) = \begin{cases}
            \left( U_\sigma(\Phi, R) \, U^{-}_\sigma(\Phi, R) \right)^{k+1} \, U_\sigma(\Phi, R) & \text{if } \Phi_\rho \sim {}^2 A_{2n} \ (n \geq 1); \\
            \left( U_\sigma(\Phi, R) \, U^{-}_\sigma(\Phi, R) \right)^2 & \text{otherwise}.
        \end{cases}
    \]
\end{thm}

\begin{rmk}
    We note that this result was previously proven by A. Smolensky~\cite{AS} for the cases of finite fields and the field of complex numbers. In the present paper, however, we establish it in significantly more general settings.
\end{rmk}


\section{Chevalley and Twisted Chevalley Groups}

In this section, we provide a formal definitions of Chevalley and twisted Chevalley groups. 
For a comprehensive treatment of these topics, we refer the reader to \cite{EA1, RC, SG&DM1, RS, RSTCG}. 
We primarily adhere to the notation introduced in \cite{SG&DM1}. 


\subsection{Chevalley Groups}

It is well known that semisimple algebraic groups over $\mathbb{C}$ (or any algebraically closed field of characteristic zero) are classified by their root systems, which determine their type, and by the lattice between the root and (fundamental) weight lattices, which determines their isogeny class.

Let $\Phi$ be a (reduced) root system, and let $\Lambda_{\text{ad}}$ and $\Lambda_{\text{sc}}$ denote the corresponding root lattice and (fundamental) weight lattice, respectively. Consider a lattice $\Lambda_{\pi}$ such that $\Lambda_{\text{ad}} \subseteq \Lambda_{\pi} \subseteq \Lambda_{\text{sc}}$. The unique (up to isomorphism) semisimple linear algebraic group over $\mathbb{C}$ associated with the pair $(\Phi, \Lambda_{\pi})$ is denoted by $G_{\pi}(\Phi, \mathbb{C})$.

Moreover, this group is defined over $\mathbb{Z}$, which allows us to define the $R$-split group $G_{\pi}(\Phi, R)$ for any commutative ring $R$ with unity. This construction yields what is known as the \textit{Chevalley group} of type $\Phi$ over $R$. When $\Lambda_{\pi} = \Lambda_{\text{ad}}$, the group is referred to as the \textit{adjoint Chevalley group}, and when $\Lambda_{\pi} = \Lambda_{\text{sc}}$, it is called the \textit{simply connected} (or \textit{universal}) Chevalley group.

For each root $\alpha \in \Phi$, there is an associated unipotent subgroup $U_{\alpha} = \{ x_{\alpha}(t) \mid t \in R \} \subseteq G_{\pi}(\Phi, R),$ which is isomorphic to the additive group $(R, +)$. The group generated by all elements $x_{\alpha}(t)$, for $\alpha \in \Phi$ and $t \in R$, is called the \textit{elementary Chevalley group}, and is denoted by $E_{\pi}(\Phi, R)$. 

Let $T = T_{\pi}(\Phi, R)$ be a split maximal torus of $G_{\pi}(\Phi, R)$. Let $B = B_{\pi}(\Phi, R)$ and $B^{-} = B^{-}_{\pi}(\Phi, R)$ be a pair of opposite Borel subgroups containing $T$. Denote by $U = U_{\pi}(\Phi, R)$ (resp., $U^{-} = U^{-}_{\pi}(\Phi, R)$) the unipotent radicals of $B$ (resp., $B^{-}$). Recall that
\[
    U = \langle x_{\alpha}(t) \mid \alpha \in \Phi^+,\, t \in R \rangle \quad \text{and} \quad U^{-} = \langle x_{\alpha}(t) \mid \alpha \in \Phi^-,\, t \in R \rangle.
\]
Moreover, we have $B = T \ltimes U$ and $B^{-} = T \ltimes U^{-}$.

For any $t \in R^*$, define
\[
    w_\alpha(t) = x_{\alpha}(t) \, x_{-\alpha}(-t^{-1}) \, x_{\alpha}(t) \quad \text{and} \quad h_{\alpha}(t) = w_\alpha(t) \, w_\alpha(1)^{-1}.
\]
Now, set
\[
    H_{\pi}(\Phi, R) = \langle h_{\alpha}(t) \mid \alpha \in \Phi,\, t \in R^* \rangle 
    \quad \text{and} \quad 
    N_{\pi} (\Phi, R) = \langle w_{\alpha}(t) \mid \alpha \in \Phi,\, t \in R^* \rangle.
\]
Then $H_{\pi}(\Phi, R)$ is a normal subgroup of $N_{\pi}(\Phi, R)$, and the quotient satisfies 
\[
    N_{\pi}(\Phi, R) / H_{\pi}(\Phi, R) \cong W(\Phi),
\]
where $W(\Phi)$ denotes the Weyl group of the root system $\Phi$.

It is known that $H_{\pi}(\Phi, R) = T_{\pi}(\Phi, R) \cap E_{\pi}(\Phi, R)$. Moreover, in the case where the group is simply connected, we have $H_{\text{sc}}(\Phi, R) = T_{\text{sc}}(\Phi, R)$.


\subsection{Twisted Chevalley Groups}

Let $R$ be a commutative ring with unity, and let $\Phi$ be an irreducible root system of type $A_n \ (n \geq 2)$, $D_n \ (n \geq 4)$, or $E_6$. Let $G = G_\pi(\Phi, R)$ (respectively, $E = E_\pi(\Phi, R)$) denote the Chevalley group (respectively, the elementary Chevalley group) over $R$ of type $\Phi$.

Fix a simple system $\Delta$ in $\Phi$. Let $\rho$ be an angle-preserving permutation of $\Delta$ (such permutations correspond to automorphisms of the associated Dynkin diagram). This permutation induces a group automorphism of $G$, called a \emph{graph automorphism}, which we also denote by $\rho$. Assume that $\rho$ is nontrivial. Then $\rho$ has order either $2$ or $3$, with the latter occurring only when $\Phi$ is of type $D_4$.

A ring automorphism $\theta: R \to R$ induces a group automorphism of $G$, also denoted by $\theta$, which is referred to as a \emph{ring automorphism} of $G$. Suppose that $R$ admits a ring automorphism $\theta$ whose order coincides with that of the graph automorphism $\rho$. 

Define $\sigma = \theta \circ \rho$; then $\sigma$ is an automorphism of $G$ of order $o(\rho) = o(\theta)$. The subgroup of $G$ fixed pointwise by $\sigma$, $G_\sigma = G_{\pi, \sigma}(\Phi, R) = \{ g \in G \mid \sigma(g) = g \}$ is called the \emph{twisted Chevalley group} over $R$ of type ${}^n X$, where $X$ denotes the type of $\Phi$ and $n$ is the order of $\sigma$.

For any subset $A \subset G$, define $A_\sigma := A \cap G_\sigma$. In particular, define $E_\sigma = E_{\pi, \sigma} (\Phi, R) := E(\Phi, R) \cap G_{\pi, \sigma} (\Phi, R)$.
Analogously, we define $U_\sigma = U_{\pi, \sigma} (\Phi, R)$, $U_\sigma^{-} = U_{\pi, \sigma}^{-}(\Phi, R)$, $T_\sigma = T_{\pi, \sigma}(\Phi, R)$, $H_\sigma = H_{\pi, \sigma}(\Phi, R)$, $N_\sigma = N_{\pi, \sigma}(\Phi, R)$ and so on.

Let $E'_\sigma = E'_{\pi, \sigma} (\Phi, R)$ denote the subgroup of $G_\sigma$ generated by $U_{\pi, \sigma}(\Phi, R)$ and $U^-_{\pi, \sigma}(\Phi, R)$. This group is referred as the \emph{elementary twisted Chevalley group} over $R$ of type ${}^n X$. 
We define $H'_\sigma = H'_{\pi, \sigma}(\Phi, R) := H(\Phi, R) \cap E'_{\pi, \sigma}(\Phi, R)$ and $N'_\sigma = N'_{\pi, \sigma}(\Phi, R) := N(\Phi, R) \cap E'_{\pi, \sigma}(\Phi, R)$.
For completeness, we also define $G'_\sigma = G'_{\pi, \sigma} (\Phi, R) = T_{\pi, \sigma} (\Phi, R) E'_{\pi, \sigma} (\Phi, R)$. 

If $G_\pi(\Phi, R)$ is of type $X$ and $\sigma$ is of order $n$, we say $G_{\pi, \sigma}(\Phi, R)$ is of type ${}^n X$. 
We write $G_\pi(\Phi, R) \sim X$ and $G_{\pi, \sigma}(\Phi, R) \sim {}^nX$. 
We use a similar notation for $E_{\pi}(\Phi, R), E_{\pi, \sigma}(\Phi, R)$ and $E'_{\pi, \sigma}(\Phi, R)$. 

\begin{conv}
    Since we will be working exclusively in the simply connected case, that is, when $\Lambda_\pi = \Lambda_{\mathrm{sc}}$, we shall omit the subscript $\pi$ from the notation $G_{\pi, \sigma}(\Phi, R)$ and write simply $G_\sigma(\Phi, R)$. This convention extends to all subgroups introduced above.
\end{conv}


\subsection{The subgroup \texorpdfstring{$E'_{\sigma} (\Phi, R)$}{E'(R)}}

We continue to use the notation introduced in the preceding subsection. For a more comprehensive treatment of the material discussed here, the reader is referred to~\cite{SG&DM1}.


\subsubsection{}

We begin by defining the \emph{twisted root system}. As before, let $\rho$ denote an angle-preserving permutation of the simple system $\Delta$. This permutation naturally extends to a permutation of the full root system $\Phi$. For each $\alpha \in \Phi$, define
\[
    \hat{\alpha} := \frac{1}{o(\rho)} \sum_{i=1}^{o(\rho)} \rho^i(\alpha),
\]
where $o(\rho)$ denotes the order of the automorphism $\rho$.

We introduce an equivalence relation $\equiv$ on $\Phi$ by declaring $\alpha \equiv \beta$ if and only if $\hat{\alpha}$ is a positive scalar multiple of $\hat{\beta}$.
Let $\Phi_\rho$ denote the set of equivalence classes in $\Phi$ under this relation. The set $\Phi_\rho$ itself forms a root system, called \textit{twisted root system}. If $\Phi$ is of type $X$ and $o(\rho) = n$, we denote the resulting twisted root system by $\Phi_\rho \sim {}^n X$.

Given an equivalence class $[\alpha] \in \Phi_\rho$, the roots in $[\alpha]$ form a positive system of roots of type either $A_1$, $A_1^2$, $A_1^3$, or $A_2$ (see Lemma~2.1 of \cite{SG&DM1}). In this case, we write $[\alpha] \sim A_1$, $A_1^2$, $A_1^3$, or $A_2$, respectively.


\subsubsection{}

We denote $\bar{\alpha} = \rho(\alpha)$, $\bar{\bar{\alpha}} = \rho^2(\alpha)$, $\bar{t} = \theta(t)$, and $\bar{\bar{t}} = \theta^2(t)$. Let $R_\theta = \{ t \in R \mid t = \bar{t} \}$ and $R^{-}_\theta = \{ t \in R \mid t = - \bar{t} \}$. We now define certain special elements of $E'_\sigma(\Phi, R)$ as follows:

\begin{enumerate}
    \item If $[\alpha] \sim A_1$ (that is, $[\alpha]=\{ \alpha \}$), then define $x_{[\alpha]}(t) = x_\alpha (t), t \in R_\theta.$ In this case, $x_{[\alpha]}(t)x_{[\alpha]}(u)=x_{[\alpha]}(t+u)$ for every $t,u \in R_\theta$.

    \item If $[\alpha] \sim A_1^2$ (that is, $[\alpha]=\{ \alpha, \bar{\alpha} \}$), then define $x_{[\alpha]}(t) = x_\alpha (t) \ x_{\bar{\alpha}} (\bar{t}), t \in R$. In this case, $x_{[\alpha]}(t)x_{[\alpha]}(u)=x_{[\alpha]}(t+u)$ for every $t,u \in R$. 
    
    \item If ${[\alpha]} \sim A_1^3$ (that is, $[\alpha]=\{ \alpha, \bar{\alpha}, \bar{\bar{\alpha}} \}$), then define $x_{[\alpha]}(t) = x_\alpha (t) \ x_{\bar{\alpha}} (\bar{t}) \ x_{\bar{\bar{\alpha}}} (\bar{\bar{t}}), \ t \in R.$ 
    In this case, $x_{[\alpha]}(t)x_{[\alpha]}(u)=x_{[\alpha]}(t+u)$ for every $t,u \in R$. 
    
    \item If ${[\alpha]} \sim A_2$ with $\alpha \neq \bar{\alpha}$ (that is, ${[\alpha]}=\{ \alpha, \bar{\alpha}, \alpha + \bar{\alpha} \}$), then define $$x_{[\alpha]}(t,u) = x_\alpha (t) x_{\bar{\alpha}} (\bar{t}) x_{\alpha + \bar{\alpha}}(N_{\bar{\alpha}, \alpha} u), \hspace{1mm} \text{where } t,u \in R \text{ such that } t \bar{t} = u + \bar{u}.$$ 
    In this case, $x_{[\alpha]}(t,u)x_{[\alpha]}(t',u')=x_{[\alpha]}(t+t', u+u' + \bar{t} t')$ for every $t,u,t',u' \in R$ such that $t \bar{t} = u + \bar{u}$ and $t' \bar{t'} = u' + \bar{u'}$.
\end{enumerate}

Define $\mathcal{A}(R) := \{ (t,u) \mid t, u \in R \text{ such that } t \bar{t} = u + \bar{u} \}$. Note that, for $[\alpha] \sim A_2$ we define $x_{[\alpha]} (t,u)$ only in the case of $(t,u) \in \mathcal{A}(R)$. The product of $x_{[\alpha]}(t,u)$ and $x_{[\alpha]}(t',u')$ suggest the operation on $\mathcal{A}(R)$ as follows: let $(t,u), (t',u') \in \mathcal{A}(R)$, then define an operation $\oplus$ on $\mathcal{A}(R)$ by $(t,u) \oplus (t',u') = (t + t', u+u'+\bar{t}t').$ With this operation $\mathcal{A}(R)$ becomes a group with $(0,0)$ as an identity and $(-t, \bar{u})$ as an inverse of $(t,u)$. From this we can say that $(x_{[\alpha]}(t,u))^{-1} = x_{[\alpha]} (-t, \bar{u})$. Further, we can define an action of the monoid $(R, \times)$ on the set $\mathcal{A}(R)$ by $$r \cdot (t,u) = (rt, r \bar{r} u)$$ for any $r \in R$ and $(t,u) \in \mathcal{A}(R)$. (For more information, please refer to \cite{EA1}.)

For $[\alpha] \in \Phi_\rho$, we write 
\begin{align*}
    R_{[\alpha]} = \begin{cases}
        R_\theta & \text{if } [\alpha] \sim A_1, \\
        R & \text{if } [\alpha] \sim A_1^2 \text{ or } A_1^3, \\
        \mathcal{A}(R) & \text{if } [\alpha] \sim A_2.
    \end{cases}
\end{align*}
If $[\alpha] \sim A_2$ then $t \in R_{[\alpha]}$ means that $t$ is a pair $(t_1, t_2)$ such that $(t_1, t_2) \in \mathcal{A}(R)$. Additionally, for $r \in R$ and $t \in R_{[\alpha]}$, the notation $r \cdot t$ means $rt$ if $[\alpha] \sim A_1, A_1^2, A_1^3$, and it means $(rt_1, r \bar{r} t_2)$ if $[\alpha] \sim A_2$.

\begin{lemma}
    The group $E'_\sigma (\Phi, R)$ is generated by $x_{[\alpha]}(t)$ for all $[\alpha] \in \Phi_\rho$ and $t \in R_{[\alpha]}$.
\end{lemma}


\subsubsection{}

Define $R^*= \{r \in R \mid \exists s \in R \text{ such that } rs = 1\}$, $R_\theta^* = R_\theta \cap R^*$ and $\mathcal{A}(R)^* := \{ (t,u) \in \mathcal{A}(R) \mid u \in R^* \}.$ 
For given $[\alpha] \in \Phi_\rho$, we write 
\[
    R_{[\alpha]}^* = \begin{cases}
        R^*_\theta & \text{if } [\alpha] \sim A_1, \\
        R^* & \text{if } [\alpha] \sim A_1^2 \text{ or } A_1^3, \\
        \mathcal{A}(R)^* & \text{if } [\alpha] \sim A_2;
    \end{cases} \quad \text{and} \quad     R_{[\alpha]}^{\star} = \begin{cases}
        R^*_\theta & \text{if } [\alpha] \sim A_1, \\
        R^* & \text{if } [\alpha] \sim A_1^2, A_1^3 \text{ or } A_2.
    \end{cases}
\]

With these notations in place, we now define certain special elements of $N_\sigma (\Phi, R)$ and $H_\sigma (\Phi, R)$:
\begin{enumerate}
    \item[$\textbf{(W1)}$] If $[\alpha] \sim A_1$, then define $w_{[\alpha]}(t) := x_{[\alpha]}(t) x_{-[\alpha]}(-t^{-1}) x_{[\alpha]}(t) = w_\alpha (t), t \in R_\theta^*.$
    \item[$\textbf{(W2)}$] If $[\alpha] \sim A_1^2$, then define $w_{[\alpha]}(t) := x_{[\alpha]}(t) x_{-[\alpha]}(-t^{-1}) x_{[\alpha]}(t) = w_{\alpha} (t) w_{\bar{\alpha}} (\bar{t}), t \in R^*.$
    \item[$\textbf{(W3)}$] If $[\alpha] \sim A_1^3$, then define $w_{[\alpha]}(t) := x_{[\alpha]}(t) x_{-[\alpha]}(-t^{-1}) x_{[\alpha]}(t) = w_{\alpha} (t) w_{\bar{\alpha}} (\bar{t}) w_{\bar{\bar{\alpha}}} (\bar{\bar{t}}), t \in R^*.$
    \item[$\textbf{(W4)}$] If $[\alpha] \sim A_2$, then define $w_{[\alpha]}(t,u) := x_{[\alpha]}(t,u) x_{-[\alpha]}( -\bar{u}^{-1} \cdot (t,u)) x_{[\alpha]}( u \bar{u}^{-1} \cdot (t,u)) = x_{[\alpha]}(t,u) x_{-[\alpha]}(- (\bar{u}^{-1}) t, (\bar{u}^{-1})) x_{[\alpha]}(u \bar{u}^{-1} t, u),$ where $(t,u) \in \mathcal{A}(R)^*$. 
    \item[\textbf{(W4$'$)}] If $[\alpha] \sim A_2$ such that $\alpha \neq \bar{\alpha}$, then define $w_{[\alpha]} (t) := w_\alpha (\bar{t}) w_{\bar{\alpha}}(1) w_{\alpha} (t), t \in R^*$. \\
    
    \item[$\textbf{(H1)}$] If $[\alpha] \sim A_1$, then define $h_{[\alpha]}(t) := w_{[\alpha]}(t) w_{[\alpha]}(-1) = h_\alpha (t), \ t \in R^*_{\theta}$.
    \item[$\textbf{(H2)}$] If $[\alpha] \sim A_1^2$, then define $h_{[\alpha]}(t) = w_{[\alpha]}(t) w_{[\alpha]}(-1) = h_\alpha (t) h_{\bar{\alpha}}(\bar{t}), \ t \in R^*$.
    \item [$\textbf{(H3)}$] If $[\alpha] \sim A_1^3$, then define $h_{[\alpha]}(t) = w_{[\alpha]}(t) w_{[\alpha]}(-1) = h_\alpha (t) h_{\bar{\alpha}}(\bar{t}) h_{\bar{\bar{\alpha}}}(\bar{\bar{t}}), \ t \in R^*$.
    \item[$\textbf{(H4)}$] If $[\alpha] \sim A_2$, then define $h_{[\alpha]}((t,u),(t',u')) = w_{[\alpha]}(t,u) w_{[\alpha]}(t',u')$, where $(t,u), (t',u') \in \mathcal{A}(R)^*$. 
    \item[\textbf{(H4$'$)}] If $[\alpha] \sim A_2$ such that $\alpha \neq \bar{\alpha}$, then define $h_{[\alpha]} (t) := h_\alpha (t) h_{\bar{\alpha}}(\bar{t}), t \in R^*$.
\end{enumerate}

\begin{rmk}
    \normalfont
    \begin{enumerate}[(a)]
        \item Let $w_{[\alpha]}(t)$, for $t \in R_{[\alpha]}^*$, be as defined in $(\mathrm{W1}), (\mathrm{W2}), (\mathrm{W3})$, or $(\mathrm{W4})$. 
        Then we have $w_{[\alpha]}(t) \in N'_\sigma(\Phi, R) \subset N_\sigma(\Phi, R)$. 
        However, the element $w_{[\alpha]}(t)$ defined in $(\mathrm{W4}')$ does not, in general, belong to $N'_\sigma(\Phi, R)$.
        
        \item Similarly, let $h_{[\alpha]} = h_{[\alpha]}(t)$, for $t \in R_{[\alpha]}^*$, be as defined in $(\mathrm{H1}), (\mathrm{H2})$, or $(\mathrm{H3})$, and let $h_{[\alpha]} = h_{[\alpha]}(t_1, t_2)$, for $t_1, t_2 \in \mathcal{A}(R)$, be as defined in $(\mathrm{H4})$. 
        Then $h_{[\alpha]} \in H'_\sigma(\Phi, R) \subset H_\sigma(\Phi, R)$. 
        However, the element $h_{[\alpha]}(t)$ defined in $(\mathrm{H4}')$ does not, in general, belong to $H'_\sigma(\Phi, R)$.
    \end{enumerate}
\end{rmk}

Since $G$ is simply connected, we have $H(\Phi, R) = T(\Phi, R)$ for every $\Phi$. It then follows that $H_{\sigma}(\Phi, R) = T_{\sigma}(\Phi, R)$ for every $\Phi_\rho$.
Furthermore, by preceding remark, we also have $H'_{\sigma}(\Phi, R) = T_{\sigma}(\Phi, R)$ when $\Phi_\rho \sim {}^2 A_{2n+1}$ $(n \geq 1)$, ${}^2 D_n$ $(n \geq 4)$, ${}^2 E_6$, or ${}^2 G_2$. 
However, it is not known whether this equality holds in the case of ${}^2 A_{2n}$ $(n \geq 1)$. In Section~\ref{sec:theta-complete ring}, we will show that under a certain class of rings, this equality indeed holds in this case.


\subsection{Closed Subsets of Twisted Root System}\label{subsec:parabolic}

Let $\Phi_\rho$ be a twisted root system of type ${}^2 A_n$ ($n \geq 2$), ${}^2 D_n$ ($n \geq 4$), ${}^2 E_6$, or ${}^3 D_4$.  
A subset \( S \subset \Phi_\rho \) is said to be \textbf{closed} if it satisfies the following conditions:
\begin{enumerate}
    \item For any $[\alpha], [\beta] \in S$ such that $[\alpha] + [\beta] \in \widetilde{\Phi}_\rho$, we have $[\alpha] + [\beta] \in S$.
  
    \item Additionally, if $\Phi_\rho \sim {}^2 A_{2n}$, then for any $[\alpha], [\beta] \in S$ of type $A_1^2$ such that $\frac{1}{2}([\alpha] + [\beta]) \in \Phi_\rho$, it follows that $\frac{1}{2}([\alpha] + [\beta]) \in S$.
\end{enumerate}

A subset $I \subset S$ of a closed set $S$ is called an \textbf{ideal} if the following conditions hold:
\begin{enumerate}
  \item For any $[\alpha] \in I$ and $[\beta] \in S$ such that $[\alpha] + [\beta] \in S$, we have $[\alpha] + [\beta] \in I$.

  \item Additionally, if $\Phi_\rho \sim {}^2 A_{2n}$, then for any $[\alpha] \in I, [\beta] \in S$ of type $A_1^2$, such that $\frac{1}{2}([\alpha] + [\beta]) \in S$, it follows that $\frac{1}{2}([\alpha] + [\beta]) \in I$.
\end{enumerate}

A closed subset $S$ of $\Phi_\rho$ is called:
\begin{enumerate}[(a)]
    \item \textbf{symmetric} (or \textbf{reductive}) if $[\alpha] \in S$ implies $-[\alpha] \in S$ (i.e., $S = -S$);
    \item \textbf{special} (or \textbf{unipotent}) if $[\alpha] \in S$ implies $-[\alpha] \notin S$ (i.e., $S \cap -S = \emptyset$).
\end{enumerate}

Any closed subset $S$ of $\Phi_\rho$ admits a unique decomposition into the disjoint union of:
\[
S^r = \{ [\alpha] \in S \mid -[\alpha] \in S \}, \quad 
S^u = \{ [\alpha] \in S \mid -[\alpha] \notin S \},
\]
called the \textbf{symmetric part} and the \textbf{special part} of $S$, respectively.
Both $S^r$ and $S^u$ are closed subsets of $\Phi_\rho$. Moreover, $S^r$ is a closed root subsystem of $\Phi_\rho$ and $S^u$ is an ideal in $S$.

Now, for a closed subset $S$ of $\Phi_\rho$, define $E'_{\sigma} (S, R)$ as the subgroup generated by all elementary root unipotent elements $x_{[\alpha]}(t)$ where $[\alpha] \in \Phi_\rho$ and $t \in R_{[\alpha]}$:
\[
    E'_{\sigma} (S, R) = \langle x_{[\alpha]} (t) \mid [\alpha] \in \Phi_\rho, t \in R_{[\alpha]} \rangle.
\]
A Levi decomposition asserts that the group $E'_{\sigma} (S, R)$ decomposes into semidirect product 
\[
    E'_\sigma (S, R) = E'_\sigma (S^r, R) \ltimes E'_\sigma (S^u, R)
\]
of its Levi subgroup $E'_\sigma (S^r, R)$ and its unipotent radical $E'_\sigma (S^u, R)$.


\section{The Special Unitary Group \texorpdfstring{$SU(3, R)$}{SU(3,R)}}

Let $R$ be a commutative ring with unity, and let $\theta$ be an involution on $R$. Let $R_\theta$, $R^{-}_{\theta}$, $\mathcal{A}(R)$, and $\mathcal{A}(R)^*$ be as defined in the previous section. For a matrix $A = (a_{ij}) \in M(n, R)$, we denote by $\bar{A}$ the matrix $(\bar{a}_{ij})$.

Let $SL(3, R)$ denote the special linear group of degree $3$ over the ring $R$. For $t \in R$ and $1 \leq i \neq j \leq 3$, define
\[
    x_{i,j}(t) = I_3 + t E_{i,j},
\]
where $E_{i,j}$ denotes the $3 \times 3$ matrix with a $1$ in the $(i,j)$-th position and $0$ elsewhere. Let $E(3, R)$ be the subgroup of $SL(3, R)$ generated by all such elementary matrices $x_{i,j}(t)$. 
Let $T(3, R) \subset SL(3, R)$ denote the subgroup of $SL(3, R)$ consisting of all diagonal matrices, and define
\[
    H(3, R) = T(3, R) \cap E(3, R).
\]

Let us consider the automorphism \(\sigma\) of \(SL(3, R)\) defined by
\[
    \sigma(A) = J \left( \bar{A}^{\,t} \right)^{-1} J, \quad \text{where} \quad J = \begin{pmatrix}
        0 & 0 & -1 \\
        0 & 1 & 0 \\
        -1 & 0 & 0
    \end{pmatrix}.
\]
The subgroup of fixed points of \(SL(3, R)\) under the automorphism \(\sigma\) is 
\[
    SU(3, R) = \left\{ A \in SL(3, R) \mid A^t J \bar{A} = J \right\}.
\]

For $(t, u) \in \mathcal{A}(R)$, define 
\[
    x_{+}(t,u) = \begin{pmatrix}
        1 & t & u \\
        0 & 1 & \bar{t} \\
        0 & 0 & 1
    \end{pmatrix} \quad \text{and} \quad x_{-}(t,u) = \begin{pmatrix}
        1 & 0 & 0 \\
        \bar{t} & 1 & 0 \\
        u & t & 1
    \end{pmatrix}.
\]
It is straightforward to verify that both $x_{+}(t, u)$ and $x_{-}(t, u)$ lie in $SU(3, R)$. Define
\[
    U^{+}_{\sigma} (3, R) = \{ x_{+}(t, u) \mid (t,u) \in \mathcal{A}(R) \} \quad \text{and} \quad U^{-}_{\sigma} (3, R) = \{ x_{-}(t, u) \mid (t,u) \in \mathcal{A}(R) \}.
\]
Then $U^{+}_{\sigma} (3, R)$ (respectively, $U^{-}_{\sigma} (3, R)$) forms a subgroup of $SU(3, R)$ consisting of all upper (respectively, lower) triangular matrices. 
Let $E'_\sigma(3, R)$ denote the subgroup of $SU(3, R)$ generated by $U^{+}_{\sigma}(3, R)$ and $U^{-}_{\sigma}(3, R)$. 
Define $E_\sigma(3, R) = E(3, R) \cap SU(3, R)$, $T_\sigma(3, R) = T(3, R) \cap SU(3, R)$, $H_\sigma(3, R) = H(3, R) \cap SU(3, R)$ and $H'_\sigma(3, R) = H(3, R) \cap E'_\sigma(3, R)$. 
It is clear that $E'_\sigma(3, R) \subset E_\sigma(3, R)$ and $H'_\sigma(3, R) \subset H_\sigma(3, R)$. 

For every $r \in R^*$, define 
\[
    h(r) = \begin{pmatrix}
        r & 0 & 0 \\
        0 & \bar{r} r^{-1} & 0 \\
        0 & 0 & \bar{r}^{-1}
    \end{pmatrix}.
\]
Then $h(r) \in SU(3, R)$ and 
\[
    T_{\sigma} (3, R) = \{ h (r) \mid r \in R^* \}.
\]
Note that $T_\sigma (3, R)$ acts on $U^{+}_\sigma (3, R)$ and $U^{-}_\sigma (3, R)$ by conjugation. Specifically, for $r \in R^*$ and $(t, u) \in \mathcal{A}(R)$, we have 
\begin{align}
    h(r) \, x_{+}(t, u) \, h(r)^{-1} &= x_{+} \left( r^2 \bar{r}^{-1} t, \, r \bar{r} u \right), \label{H1} \tag{H1} \\
    h(r) \, x_{-}(t, u) \, h(r)^{-1} &= x_{-} \left( r \bar{r}^{-2} t, \, r^{-1} \bar{r}^{-1} u \right). \label{H2} \tag{H2}
\end{align}

Next, for every $r \in R^*$, define 
\[
    w(r) = \begin{pmatrix}
        0 & 0 & r \\
        0 & -r^{-1}\bar{r} & 0 \\
        \bar{r}^{-1} & 0 & 0
    \end{pmatrix}.
\]
Then we have $w(r) \in SU(3, R)$. For $r \in R^*$ and $(t,u) \in \mathcal{A}(R)$, we have
\begin{align}
    w(r) \, x_{+}(t, u) \, w(r)^{-1} &= x_{-} \left( -r \bar{r}^{-2} t, \, (r \bar{r})^{-1} u \right), \label{W1} \tag{W1} \\
    w(r) \, x_{-}(t, u) \, w(r)^{-1} &= x_{+} \left( -r^2 \bar{r}^{-1} t, \, r\bar{r} u \right). \label{W2} \tag{W2}
\end{align}
Moreover, for $r,s \in R^*$ we have
\begin{gather}
    h(r) = w (r) w(1)^{-1}, \label{HW1} \tag{HW1} \\
    h(r) \, w(s) \, h(r)^{-1} = w(r \bar{r} s), \label{HW2} \tag{HW2} \\
    w(r) \, h(s) \, w(r)^{-1}= h(\bar{s}^{-1}). \label{HW3} \tag{HW3} 
\end{gather}

Given any $(t, u) \in \mathcal{A}(R)^*$, define
\[
    w_{\pm}(t, u) := x_{\pm}(t, u) \cdot x_{\mp}(-tu^{-1}, \bar{u}^{-1}) \cdot x_{\pm}(tu^{-1} \bar{u}, u) \in E'_\sigma(3, R).
\]
For any pair $(t_1, u_1), (t_2, u_2) \in \mathcal{A}(R)^*$, define
\[
    h_{\pm}((t_1, u_1), (t_2, u_2)) := w_{\pm}(t_1, u_1) \cdot w_{\pm}(t_2, u_2) \in E'_\sigma(3, R).
\]
Observe that 
\begin{equation}
    w_{+}(t, u) = w(u) \quad \text{and} \quad w_{-}(t, u) = w(\bar{u}^{-1}). \label{W3} \tag{W3}
\end{equation}
Furthermore, we have
\begin{equation}
    h_{+}((t_1, u_1), (t_2, u_2)) = h(u_1 \bar{u}_2^{-1}) \quad \text{and} \quad h_{-}((t_1, u_1), (t_2, u_2)) = h(\bar{u}_1^{-1} u_2). \label{H3} \tag{H3}
\end{equation}

Let $A = (a_{ij}) \in SU(3, R)$. The following are the independent conditions on the entries of $A$ implied by the relation
\[
    A^t J \, \bar{A} = J,
\]
with all redundant equations omitted:
\begin{align}
    a_{31} \overline{a_{11}} - a_{21} \overline{{a}_{21}} + a_{11} \overline{{a}_{31}} &= 0, \label{U11} \tag{U-A.1.1} \\
    a_{31} \overline{{a}_{12}} - a_{21} \overline{{a}_{22}} + a_{11} \overline{{a}_{32}} &= 0, \label{U12} \tag{U-A.1.2} \\
    a_{31} \overline{{a}_{13}} - a_{21} \overline{{a}_{23}} + a_{11} \overline{{a}_{33}} &= 1, \label{U13} \tag{U-A.1.3} \\
    a_{32} \overline{{a}_{12}} - a_{22} \overline{{a}_{22}} + a_{12} \overline{{a}_{32}} &= -1, \label{U22} \tag{U-A.2.2} \\
    a_{32} \overline{{a}_{13}} - a_{22} \overline{{a}_{23}} + a_{12} \overline{{a}_{33}} &= 0, \label{U23} \tag{U-A.2.3} \\
    a_{33} \overline{{a}_{13}} - a_{23} \overline{{a}_{23}} + a_{13} \overline{{a}_{33}} &= 0. \label{U33} \tag{U-A.3.3}
\end{align}

An alternative set of relations for the same matrix can be derived from the identity
\[
    J \, \bar{A} = (A^t)^{-1} J.
\]
These relations (with redundancies) are listed below:
\begin{align}
    \overline{a_{11}} &= a_{11}a_{22} - a_{12}a_{21}, \label{U'11} \tag{U-B.1.1} \\
    \overline{a_{12}} &= a_{11}a_{23} - a_{13}a_{21}, \label{U'12} \tag{U-B.1.2} \\
    \overline{a_{13}} &= a_{12}a_{23} - a_{13}a_{22}, \label{U'13} \tag{U-B.1.3} \\
    \overline{a_{21}} &= a_{11}a_{32} - a_{12}a_{31}, \label{U'21} \tag{U-B.2.1} \\
    \overline{a_{22}} &= a_{11}a_{33} - a_{13}a_{31}, \label{U'22} \tag{U-B.2.2} \\
    \overline{a_{23}} &= a_{12}a_{33} - a_{13}a_{32}, \label{U'23} \tag{U-B.2.3} \\
    \overline{a_{31}} &= a_{21}a_{32} - a_{22}a_{31}, \label{U'31} \tag{U-B.3.1} \\
    \overline{a_{32}} &= a_{21}a_{33} - a_{23}a_{31}, \label{U'32} \tag{U-B.3.2} \\
    \overline{a_{33}} &= a_{22}a_{33} - a_{23}a_{32}. \label{U'33} \tag{U-B.3.3}
\end{align}


\section{Special Stable Range One Conditions}\label{sec:SSR1}

Let $R$ be a commutative ring with an involution $\theta$. Recall that a vector $(a, b, c) \in R^3$ is called \emph{special unitary completable} if it appears as the first row of some matrix in $SU(3, R)$. We say that $R$ satisfies the \emph{special stable range one condition} $(SSR_1)$ if, for every such vector, there exists $(z_1, z_2) \in \mathcal{A}(R)$ such that $a + b z_1 + c z_2$ is a unit in $R$.

\begin{thm}\label{thm:SSR_1}
    \normalfont
    Let $R$ be a commutative ring with an involution $\theta$. The following conditions are equivalent:
    \begin{enumerate}[(a)]
        \item $R$ satisfies $(SSR_1)$ condition.
        \item $SU(3,R) = T_{\sigma}(3, R)\, U^{-}_{\sigma}(3, R)\, U^{+}_{\sigma}(3, R)\, U^{-}_{\sigma}(3, R)$.
        \item $SU(3,R) = T_{\sigma}(3, R)\, U^{+}_{\sigma}(3, R)\, U^{-}_{\sigma}(3, R)\, U^{+}_{\sigma}(3, R)$.
    \end{enumerate}
\end{thm}

\begin{proof}
    \noindent $(a) \implies (b):$ Let $A = (a_{ij}) \in SU(3, R)$. Then the vector $(a_{11}, a_{12}, a_{13})$ is a special unitary completable. Since $R$ satisfies $(SSR_1)$, there exists $(z_1, z_2) \in \mathcal{A}(R)$ such that $a_{11} + a_{12} z_1 + a_{13} z_2 \in R^*$.
    Define the matrix $B = (b_{ij})$ by
    \[
        B = A \, x_{-}(\bar{z}_1, z_2).
    \]
    Since $(z_1, z_2) \in \mathcal{A}(R)$, it follows that $(\bar{z}_1, z_2) \in \mathcal{A}(R)$, and hence $x_{-}(\bar{z}_1, z_2) \in SU(3, R)$. Therefore, $B \in SU(3, R)$ as well. Moreover, we have $b_{11} = a_{11} + a_{12} z_1 + a_{13} z_2 \in R^*$.

    Now consider the matrix $C = (c_{ij})$ defined by
    \[
        C = x_{-} \left( -\frac{\overline{b_{21}}}{\overline{b_{11}}},\ -\frac{b_{31}}{b_{11}} + \frac{b_{21} \overline{{b}_{21}}}{b_{11} \overline{{b}_{11}}} \right) B.
    \]
    By (\ref{U11}), the pair $\left( -\frac{\overline{b_{21}}}{\overline{b_{11}}},\ -\frac{b_{31}}{b_{11}} + \frac{b_{21} \overline{{b}_{21}}}{b_{11} \overline{{b}_{11}}} \right) \in \mathcal{A}(R)$. Hence the matrix 
    \[
    x_{-} \left( -\frac{\overline{b_{21}}}{\overline{b_{11}}},\ -\frac{b_{31}}{b_{11}} + \frac{b_{21} \overline{{b}_{21}}}{b_{11} \overline{{b}_{11}}} \right) \in SU(3, R),
    \]
    and therefore $C \in SU(3, R)$.

    Observe that $c_{21} = c_{31} = 0$, and from equation (\ref{U'32}), it follows that $c_{32} = 0$ as well. Applying equations (\ref{U13}), (\ref{U22}) and the fact that $\det(C) = 1$, the matrix $C$ takes the form
    \[
        C = \begin{pmatrix}
            c_{11} & * & * \\
            0 & (c_{11})^{-1} \bar{c}_{11} & * \\
            0 & 0 & (\bar{c}_{11})^{-1}
        \end{pmatrix}
        = h(c_{11})\, x_{+}(*,*),
    \]
    where $x_{+}(*,*) \in U^{+}_{\sigma}(3, R)$.

    Therefore, we conclude that
    \[
        x_{-} \left( -\frac{\overline{b_{21}}}{\overline{b_{11}}},\ -\frac{b_{31}}{b_{11}} + \frac{b_{21} \overline{{b}_{21}}}{b_{11} \overline{{b}_{11}}} \right) \cdot A \cdot x_{-}(\bar{z}_1, z_2) = h(c_{11})\, x_{+}(*,*).
    \]
    Rewriting $A$ accordingly, we obtain:
    \begin{align*}
        A &= x_{-}(*,*) \cdot h(*) \cdot x_{+}(*,*) \cdot x_{-}(*,*) \\
        &= h(*) \cdot \left(h(*)^{-1} x_{-}(*,*) h(*)\right) \cdot x_{+}(*,*) \cdot x_{-}(*,*) \\
        &= h(*) \cdot x_{-}(*,*) \cdot x_{+}(*,*) \cdot x_{-}(*,*) \\
        &\in T_{\sigma}(3, R)\, U^{-}_{\sigma}(3, R)\, U^{+}_{\sigma}(3, R)\, U^{-}_{\sigma}(3, R),
    \end{align*}
    as required.
    
    \vspace{2mm}

    \noindent $(b) \implies (a):$ Let $(a, b, c) \in R^3$ be a special unitary completable vector. Then, by definition, there exists a matrix $A \in SU(3, R)$ of the form
    \[
        A = \begin{pmatrix}
            a & b & c \\
            * & * & * \\
            * & * & *
        \end{pmatrix}.
    \]
    By assumption, there exists $r \in R^*$ and $(z_1, z_2) \in \mathcal{A}(R)$ such that $A$ admits the factorization
    \[
        A = h(r)\, x_{-}(*,*)\, x_{+}(*,*)\, x_{-}(-\bar{z}_1, \bar{z_2}).
    \]
    Multiplying both sides on the right by $x_{-}(\bar{z}_1, z_2)$ yields
    \[
        A \, x_{-}(\bar{z}_1, z_2) = h(r) \, x_{-}(*,*) \, x_{+}(*,*).
    \]
    Comparing the $(1,1)$-entry on both sides, we obtain
    \[
        a + b z_1 + c z_2 = r \in R^*,
    \]
    as required.

    \vspace{2mm}

    \noindent $(b) \iff (c):$ The equivalence of these statements follows directly from the following identities:
    \begin{align*}
        w(1) \, SU(3, R) \, w(1)^{-1} &= SU(3, R) && \text{since } w(1) \in SU(3, R), \\
        w(1) \, T_{\sigma}(3, R) \, w(1)^{-1} &= T_{\sigma}(3, R) && \text{by equation (\ref{HW3})}, \\
        w(1) \, U^{+}_{\sigma}(3, R) \, w(1)^{-1} &= U^{-}_{\sigma}(3, R) && \text{by equation (\ref{W1})}, \\
        w(1) \, U^{-}_{\sigma}(3, R) \, w(1)^{-1} &= U^{+}_{\sigma}(3, R) && \text{by equation (\ref{W2})}.
    \end{align*}
    This completes the proof.
\end{proof}


\begin{cor}\label{cor:SSR_1}
    \normalfont
    Let $R$ be a commutative ring with an involution $\theta$. The following are equivalent:
    \begin{enumerate}[(a)]
        \item $R$ satisfies $(SSR_1)$ condition.
        \item For every vector $(a,b,c) \in R^3$ that appears as the first column of some matrix in \( SU(3, R) \), there exists a pair $(z_1, z_2) \in \mathcal{A}(R)$ such that \( a + b z_1 + c z_2 \in R^* \).
        \item For every vector $(a,b,c) \in R^3$ that appears as the last row of some matrix in \( SU(3, R) \), there exists a pair $(z_1, z_2) \in \mathcal{A}(R)$ such that \( c + b z_1 + a z_2 \in R^* \).
        \item For every vector $(a,b,c) \in R^3$ that appears as the last column of some matrix in \( SU(3, R) \), there exists a pair $(z_1, z_2) \in \mathcal{A}(R)$ such that \( c + b z_1 + a z_2 \in R^* \).
    \end{enumerate}
\end{cor}

\begin{proof}
    \noindent $(a) \implies (b):$ Suppose $R$ satisfies $(SSR_1)$ condition. Let $(a,b,c) \in R^3$ be such that it forms the first column of some matrix $A \in SU(3, R)$. By Theorem~\ref{thm:SSR_1}, there exists $r \in R^*$ and $(z_1, z_2) \in \mathcal{A}(R)$ such that $A$ admits the factorization
    \[
        A = h(r) x_{+}(- r^{-2} \bar{r} z_1, (r \bar{r})^{-1} \bar{z}_2) x_{-}(*, *) x_{+}(*, *).
    \]
    Left-multiplying both sides by \( x_{+}(z_1, z_2) \), we get
    \begin{align*}
        x_{+}(z_1, z_2) A &= x_{+}(z_1, z_2) \left( h(r) x_{+}(- r^{-2} \bar{r} z_1, (r \bar{r})^{-1} \bar{z}_2) h(r)^{-1} \right) h(r) \, x_{-}(*, *) \, x_{+}(*, *) \\
        &= h(r) x_{-} (*, *) x_{+} (*, *).
    \end{align*}
    Looking at the \((1,1)\)-entry of both sides, we find
    \[
        a + b z_1 + c z_2 = r \in R^*,
    \]
    as desired.

    \vspace{2mm}

    \noindent $(b) \implies (a):$ By following the similar argument used in the $(a) \implies (b)$ direction of Theorem~\ref{thm:SSR_1}, we can obtain  
    \[
        SU(3, R) = T_{\sigma}(3, R) \, U^{+}_{\sigma}(3, R) \, U^{-}_{\sigma}(3, R) \, U^{+}_{\sigma}(3, R).
    \] 
    Since $(a)$ and $(c)$ of Theorem~\ref{thm:SSR_1} are equivalent, the conclusion follows.    
    
    \vspace{2mm}

    \noindent $(a) \iff (c) \ \text{and} \ (a) \iff (d):$ The proofs are analogous to that of $(a) \iff (b)$ and are therefore omitted.
\end{proof}


Proposition~\ref{prop:semilocal has SSR_1} (stated below) illustrates a class of rings that satisfy the $(SSR_1)$ condition. Before presenting the proposition, we introduce some relevant notation.
For any $r \in R$, define
\[
    M_r := \{ \mathfrak{m} \in \mathrm{Max}(R) \mid r \in \mathfrak{m} \}.
\]
Note that $M_0 = \mathrm{Max}(R)$ and $M_1 = \emptyset$.
Now, define the set
\[
    J_r := \left( \bigcap_{\mathfrak{m} \in M_0 \setminus M_r} \mathfrak{m} \right) \setminus \left( \bigcup_{\mathfrak{m} \in M_r} \mathfrak{m} \right).
\]
Adopting the standard conventions that the intersection over the empty set is the entire ring $R$, and the union over the empty set is the empty set, we observe the following equivalences:
\begin{enumerate}[(a)]
    \item $r \in R^* \quad \Longleftrightarrow \quad M_r = \emptyset \quad \Longleftrightarrow \quad J_r = \mathrm{Jac}(R)$, the Jacobson radical of $R$.
    \item $r \in \mathrm{Jac}(R) \quad \Longleftrightarrow \quad M_r = M_0 \quad \Longleftrightarrow \quad J_r = R^*.$
\end{enumerate}


\begin{lemma}\label{lemma:I_r and J_r}
    \normalfont
    Let $R$ be a commutative ring with unity, and let $r \in R$.
    \begin{enumerate}[(a)]
        \item If $J_r \neq \emptyset$, then for any $s \in J_r$, we have $r + s \in R^*$.
        \item Assume that $R$ is semi-local ring. Then $J_r \neq \emptyset$.
    \end{enumerate}
\end{lemma}

\begin{proof}
    \noindent (a) Let $s \in J_r$. Suppose, for the sake of contradiction, that $r + s \notin R^*$. 
    First, assume $r + s \in \bigcup_{\mathfrak{m} \in M_r} \mathfrak{m}$, that is, $r + s \in \mathfrak{m}_0$ for some $\mathfrak{m}_0 \in M_r$. Since $r \in \mathfrak{m}_0$, it follows that $s \in \mathfrak{m}_0$ as well. This contradicts the fact that $s \notin \bigcup_{\mathfrak{m} \in M_r} \mathfrak{m}$.

    Now assume $r + s \notin \bigcup_{\mathfrak{m} \in M_r} \mathfrak{m}$, that is, $r + s \notin \mathfrak{m}$ for every $\mathfrak{m} \in M_r$. However, since \( r + s \notin R^* \), it must lie in some maximal ideal \( \mathfrak{m}_0 \in M_0 \setminus M_r \). Since $s \in J_r$, we have $s \in \mathfrak{m}_0$, which implies that $r \in \mathfrak{m}_0$. This contradicts the fact that \( \mathfrak{m}_0 \in M_0 \setminus M_r \). 
    Therefore, $r + s \in R^*$, as desired.

    \medskip

    \noindent (b) Suppose, for contradiction, that $J_r = \emptyset$. Then, by definition,
    \[
        \bigcap_{\mathfrak{m} \in M_0 \setminus M_r} \mathfrak{m} \subseteq \bigcup_{\mathfrak{m} \in M_r} \mathfrak{m}.
    \]
    By the prime avoidance theorem, there exists $\mathfrak{m}_0 \in M_r$ such that
    \[
        \bigcap_{\mathfrak{m} \in M_0 \setminus M_r} \mathfrak{m} \subseteq \mathfrak{m}_0.
    \]
    For each $\mathfrak{m} \in M_0 \setminus M_r$, we have $\mathfrak{m} \neq \mathfrak{m}_0$, and hence $\mathfrak{m} + \mathfrak{m}_0 = R$. Therefore, for each such $\mathfrak{m}$, there exist elements $x_{\mathfrak{m}} \in \mathfrak{m}$ and $y_{\mathfrak{m}} \in \mathfrak{m}_0$ such that $x_{\mathfrak{m}} + y_{\mathfrak{m}} = 1$. 
    Define
    \[
        x := \prod_{\mathfrak{m} \in M_0 \setminus M_r} x_{\mathfrak{m}} = \prod_{\mathfrak{m} \in M_0 \setminus M_r} (1 - y_{\mathfrak{m}}) = 1 - y,
    \]
    for some $y \in \mathfrak{m}_0$. Since $x_{\mathfrak{m}} \in \mathfrak{m}$ for each $\mathfrak{m} \in M_0 \setminus M_r$, it follows that $x \in \bigcap_{\mathfrak{m} \in M_0 \setminus M_r} \mathfrak{m} \subseteq \mathfrak{m}_0$. However, $x = 1 - y$ with $y \in \mathfrak{m}_0$ implies $1 \notin \mathfrak{m}_0$, a contradiction. 
    Thus, the assumption $J_r = \emptyset$ is false, and we conclude that $J_r \neq \emptyset$, which proves $(b)$.
\end{proof}


\begin{prop}\label{prop:semilocal has SSR_1}
    \normalfont
    Let $R$ be a semi-local ring with an involution $\theta$. Assume that $\bar{\mathfrak{m}} = \mathfrak{m}$ for every maximal ideal $\mathfrak{m} \in \mathrm{Max}(R)$ and $\mathcal{A}(R)^* \neq \emptyset$. Then $R$ satisfies $(SSR_1)$ condition.
\end{prop}

\begin{proof}
    Let $(a,b,c) \in R^3$ be a special unitary completable vector. Then, by definition, there exists a matrix $A \in SU(3, R)$ of the form 
    \[
        A = \begin{pmatrix}
            a & b & c \\
            * & * & * \\
            * & * & *
        \end{pmatrix}.
    \]
    From equations~(\ref{U'11}), (\ref{U'12}) and (\ref{U'21}), we deduce that $M_a \subset M_b$. Similarly, equations~(\ref{U'12}), (\ref{U'13}) and (\ref{U'23}) imply that $M_c \subset M_b$. Consequently, $(M_a \cup M_c) \subset M_b \subset M_0$. 
    Moreover, since $\det (A) = 1$, we can deduce that $M_a \cap M_c = \emptyset$. 
    Finally, observe that $M_{ac} = M_a \cup M_c$. 

    Since $\mathcal{A}(R)^* \neq \emptyset$, choose any $(z_1, z_2) \in \mathcal{A}(R)^*$.
    Define the set
    \[
        J = \left( \bigcap_{\mathfrak{m} \in M_0 \setminus M_{ac}} \mathfrak{m} \right) \setminus \left( \bigcup_{\mathfrak{m} \in M_a} \mathfrak{m} \right).
    \]
    Note that, $J_{ac} \subset J$. By Lemma~\ref{lemma:I_r and J_r}, part $(b)$, we have $J_{ac} \neq \emptyset$ and hence $J \neq \emptyset$. 
    Fix any $z \in J$. Then $(zz_1, z \bar{z} z_2) \in \mathcal{A}(R)$.
    We observe that $c z \bar{z} z_2 \in J_a$. By Lemma~\ref{lemma:I_r and J_r}, part $(a)$, we conclude that $a + c z \bar{z} z_2 \in R^*$. Also, observe that $b z z_1 \in \mathrm{Jac} (R)$. Hence $a + b z z_1 + c z \bar{z} z_2 \in R^*$.
\end{proof}


\begin{cor}
    \normalfont
    Let $(R, \mathfrak{m}, k)$ be a local ring with an involution $\theta$. Assume that $\mathcal{A}(R)^* \neq \emptyset$. Then $R$ satisfies $(SSR_1)$ condition.
\end{cor}

\begin{cor}
    \normalfont
    Every field (with an involution) satisfies $(SSR_1)$ condition.
\end{cor}


\begin{prop}\label{prop:prod of SSR1}
    \normalfont
    Let $\Gamma$ be a fixed index set, and for each $i \in \Gamma$, let $R_i$ be a ring equipped with an involution $\theta_i$. Suppose that each $R_i$ satisfies the $(SSR_1)$ condition. Then the product ring $R = \prod_{i \in \Gamma} R_i$, equipped with the componentwise involution $\theta = \prod_{i \in \Gamma} \theta_i$, also satisfies the $(SSR_1)$ condition.
\end{prop}

\begin{proof}
    Since $R = \prod_{i \in \Gamma} R_i$, it follows that
    \[
        SU(3, R) = \prod_{i \in \Gamma} SU(3, R_i).
    \]
    For any $r \in R$, let $r_i \in R_i$ denote its $i$-th component.  
    If $(a, b, c) \in R^3$ is special unitary completable, then each $(a_i, b_i, c_i) \in R_i^3$ is also special unitary completable.
    Thus, for each $i$, there exist $(z_1^{(i)}, z_2^{(i)}) \in \mathcal{A}(R_i)$ such that $a_i + b_i z_1^{(i)} + c_i z_2^{(i)} \in R_i^*$. 
    Define $z_1 = (z_1^{(i)})_{i \in \Gamma} \quad \text{and} \quad z_2 = (z_2^{(i)})_{i \in \Gamma}$. 
    Then $(z_1, z_2) \in \mathcal{A}(R)$, and $a + b z_1 + c z_2 \in R^*$, as required.
\end{proof}


\section{\texorpdfstring{$\theta$}{theta}-complete Ring}\label{sec:theta-complete ring}

Let $R$ be a commutative ring with unity with an involution $\theta$.
As in Section~\ref{sec:intro}, we introduce the set
\[
\mathcal{B}_1(R) = \{ u \in R \mid \text{there exists } t \in R \text{ such that } (t,u) \in \mathcal{A}(R)^{*}\}.
\]
If $\mathcal{A}(R)^* \neq \emptyset$ (equivalently, $\mathcal{B}_1(R) \neq \emptyset$), then for any $k \in \mathbb{N}$, define
\[
\mathcal{B}_k(R) = \{ u \in R \mid \text{there exist } u_1, \dots, u_k \in \mathcal{B}_1(R) \text{ such that } u = u_1 \dots u_k \}.
\]
Otherwise, set $\mathcal{B}_k(R) = \emptyset$.  
Now, define the sets  
\[
\mathcal{C}_{\text{even}}(R) = \bigcup_{\ell \in \mathbb{N}} \mathcal{B}_{2\ell}(R)
\quad \text{and} \quad
\mathcal{C}_{\text{odd}}(R) = \bigcup_{\ell \in \mathbb{N}} \mathcal{B}_{2\ell - 1}(R).
\]
For the sake of completeness (though this will not be needed later) we define, for each $k \in \mathbb{N}$,  
\[
\mathcal{C}_k(R) = \bigcup_{\ell \in \mathbb{N}} \mathcal{B}_{k\ell}(R).
\]
It is evident that  
\[
\mathcal{C}_{\text{even}}(R) = \mathcal{C}_2(R).
\]

\begin{rmk}
    \normalfont
    The following properties follow immediately from the definitions:
    \begin{enumerate}[(a)]
        \item If $2 \in R^*$, then $\mathcal{B}_k(R) \neq \emptyset$ for all $k \in \mathbb{N}$, since $(2,2) \in \mathcal{A}(R)^*$. Consequently, $\mathcal{C}_k(R) \neq \emptyset$ for all $k \in \mathbb{N}$.
        
        \item \label{rmk:(b)} If $u$ is in $\mathcal{B}_k(R)$, then so are $\bar{u}, u^{-1}$ and $a \bar{a} u \ (a \in R^*)$.
        
        \item If $\mathcal{A}(R) \neq \emptyset$, then each $\mathcal{C}_k(R)$ forms a subgroup of the multiplicative group $R^*$ generated by $\mathcal{B}_k(R)$. In particular, $\mathcal{C}_{\text{even}}(R)$ is a subgroup of $R^*$.
        
        \item If $1 \in \mathcal{B}_k(R)$, then for all $m \in \mathbb{N}$, we have the inclusion  
        \[
        \mathcal{B}_m(R) \subset \mathcal{B}_{k+m}(R).
        \]
        In particular, since $1 \in \mathcal{B}_2(R)$ (provided $\mathcal{A}(R)^* \neq \emptyset$), it follows that  
        \[
        \mathcal{B}_1(R) \subset \mathcal{B}_3(R) \subset \mathcal{B}_5(R) \subset \cdots 
        \quad \text{and} \quad
        \mathcal{B}_2(R) \subset \mathcal{B}_4(R) \subset \mathcal{B}_6(R) \subset \cdots.
        \]
    \end{enumerate}
\end{rmk}

\begin{defn}
    \normalfont
    Let $R$ be a commutative ring with an involution $\theta$.
    \begin{enumerate}[(a)]
        \item We say that $R$ is \emph{$\theta$-complete} if
        \[
            \mathcal{C}_{\text{even}}(R) = R^*.
        \]
        
        \item We say that $R$ has \emph{finite $\mathcal{C}$-length} if there exists $k \in \mathbb{Z}^+$ such that
        \[
            \mathcal{C}_{\text{even}}(R) = \mathcal{B}_{2k}(R).
        \]
        The minimal such $k$ is called the \emph{$\mathcal{C}$-length} of $R$. If no such $k$ exists, we say that the $\mathcal{C}$-length of $R$ is $\infty$.
        
    \end{enumerate}
\end{defn}


\begin{prop}\label{prop:field is theta-complete}
    \normalfont
    Every field is $\theta$-complete with $\mathcal{C}$-length $1$. 
\end{prop}

\begin{proof}
    Let $R$ be a field and let $u \in R^*$. 
    Suppose that $u \neq \bar{u}$. Then the element $u - \bar{u}$ is invertible in $R$. 
    In this case, define 
    \[
        u_1 = u - \bar{u} \quad \text{and} \quad u_2 = u (u_1)^{-1}.
    \]
    Clearly, we have $u = u_1 u_2$ with $u_1, u_2 \in \mathcal{B}_1(R)$, as $(0, u_1), (1, u_2) \in \mathcal{A} (R)^{*}$, respectively. Thus, $u \in \mathcal{B}_2(R)$, as desired. 
    
    Now, suppose that $u = \bar{u}$. Choose any element $a \in R^{-}_{\theta} \cap R^*$ (such an element always exist in the field) and set
    \[
        u_1 = a \quad \text{and} \quad u_2 = u (a)^{-1}.
    \]
    Again, we observe that $u = u_1 u_2$ with $u_1, u_2 \in \mathcal{B}_1(R)$, as $(0, u_1), (0, u_2) \in \mathcal{A} (R)^{*}$, respectively. 
    Hence, we conclude that $u \in \mathcal{B}_2(R)$,  completing the proof.
\end{proof}


\begin{prop}
    \normalfont
    Let $(R, \mathfrak{m}, k)$ be a local ring. Assume that $2 \in R^*$ and that $R_\theta^{-} \cap R^* \neq \emptyset$. Then $R$ is $\theta$-complete with $\mathcal{C}$-length $\leq 2$.
\end{prop}

\begin{proof}
    Let $u \in R^*$. 
    Suppose that $u - \bar{u} \not \in \mathfrak{m}$. Then it is invertible in $R$. 
    In this case, define 
    \[
        u_1 = u - \bar{u} \quad \text{and} \quad u_2 = u (u_1)^{-1}.
    \]
    Clearly, we have $u = u_1 u_2$ with $u_1, u_2 \in \mathcal{B}_1(R)$, as $(0, u_1), (1, u_2) \in \mathcal{A} (R)^{*}$, respectively. Hence, we have $u \in \mathcal{B}_2(R)$, as desired. 
    
    Now, suppose that $u - \bar{u} \in \mathfrak{m}$. Then we must have that $u + \bar{u}$ is invertible in $R$. Choose any element $a \in R^{-}_{\theta} \cap R^*$ and set
    \[
        u_1 = a, \quad u_2 = \frac{u + \bar{u}}{2a}, \quad u_3 = \frac{u}{u + \bar{u}} \quad \text{and} \quad u_4 = 2.
    \]
    We observe that $u = u_1 u_2 u_3 u_4$ with $u_1, u_2, u_3, u_4 \in \mathcal{B}_1(R)$, as $(0, u_1), (0, u_2), (1, u_3), (2,2) \in \mathcal{A} (R)^{*}$, respectively. 
    Hence, we conclude that $u \in \mathcal{B}_4(R)$,  completing the proof.
\end{proof}


\begin{prop}
    \normalfont
    Let $R$ be a semi-local ring with exactly two maximal ideals $\mathfrak{m}_1$ and $\mathfrak{m}_2$. Suppose that $2 \in R^*$, and $R_\theta^{-} \cap R^* \neq \emptyset$. Further, assume that $\bar{\mathfrak{m}}_1 = \mathfrak{m}_2$. Then $R$ is $\theta$-complete with $\mathcal{C}$-length $\leq 2$.
\end{prop}

\begin{proof}
    Let $u \in R^*$.  
    If $u - \bar{u} \in R^*$, we proceed exactly as in the proof of the previous proposition.  
    Now, suppose $u - \bar{u} \in \mathfrak{m}_1 \cup \mathfrak{m}_2$.  
    By assumption, $u - \bar{u} \in \mathfrak{m}_1 \cap \mathfrak{m}_2 = \mathrm{Jac}(R)$.  
    In that case, $u + \bar{u} \in R^*$.  
    Thus, we are again in the situation of the previous theorem, and the same argument applies.
\end{proof}


\begin{prop} 
    \normalfont
    Let $\Gamma$ be a fixed index set, and for each $i \in \Gamma$, let $R_i$ be a ring equipped with an involution $\theta_i$. Suppose that each $R_i$ is $\theta_i$-complete with $\mathcal{C}$-length $k_i$. The the product ring $R = \prod_{i \in \Gamma} R_i$, equipped with the componentwise involution $\theta = \prod_{i \in \Gamma} \theta_i$, is $\theta$-complete with $\mathcal{C}$-length $k = \text{Sup}_{i \in \Gamma} k_i$.
\end{prop}

\begin{proof}
    The proof is similar to that of Proposition~\ref{prop:prod of SSR1} and is therefore omitted.
\end{proof}


\begin{prop}\label{prop:T=H=H'}
    \normalfont
    Let $R$ be a $\theta$-complete ring. Then 
    \[
        T_{\sigma}(3, R) = H_{\sigma}(3, R) = H'_{\sigma}(3, R).
    \]
\end{prop}

\begin{proof}
    Since $H'_\sigma(3, R) \subset H_\sigma(3, R) \subset T_\sigma(3, R)$, it suffices to show that $T_\sigma(3, R) \subset H'_\sigma(3, R)$. By (\ref{H3}) and part~\ref{rmk:(b)} of the remark preceding Proposition~\ref{prop:field is theta-complete}, we know that for every $r \in \mathcal{B}_2(R)$, we have $h(r) \in H'_\sigma(3, R)$. 

    Now, since $R$ is $\theta$-complete, for any $r \in R^*$ there exists $k \in \mathbb{Z}^+$ such that $r \in \mathcal{B}_{2k}(R)$. By definition of $\mathcal{B}_{2k}(R)$, we can express $r$ as $r = r_1 \cdots r_k$ with each $r_i \in \mathcal{B}_2(R)$. It follows that
    \[
        h(r) = h(r_1) \cdots h(r_k),
    \]
    and since each $h(r_i) \in H'_\sigma(3, R)$, we conclude that $h(r) \in H'_\sigma(3, R)$. Thus, $T_\sigma(3, R) \subset H'_\sigma(3, R)$, completing the proof.
\end{proof}


\begin{thm}\label{thm:T is as unitriangular}
    \normalfont
    Let $R$ be a $\theta$-complete ring with $\mathcal{C}$-length $k$. Then
    \[
        T_{\sigma}(3, R) \subset \left( U^{+}_\sigma (3, R) \, U^{-}_\sigma (3, R) \right)^{k+1} \quad \text{and} \quad T_{\sigma}(3, R) \subset \left( U^{-}_\sigma (3, R) \, U^{+}_\sigma (3, R) \right)^{k+1}.
    \]
\end{thm}

\begin{proof}
    Let $h(r) \in T_\sigma(3, R)$ for some $r \in R^*$. Since $R$ is $\theta$-complete, there exists $k \in \mathbb{Z}^+$ such that $r \in \mathcal{B}_{2k}(R)$. Choose the minimal such $k$. By definition of $\mathcal{B}_{2k}(R)$, there exist elements $r_1, \dots, r_{2k} \in \mathcal{B}_1(R)$ such that
    \[
        r = r_1 r_2 \cdots r_{2k}.
    \]

    Since each $r_{2i} \in \mathcal{B}_1(R)$, it follows that $r_{2i}^{-1} \in \mathcal{B}_1(R)$. Therefore, there exists $(z_1^{(i)}, z_2^{(i)}) \in \mathcal{A}(R)^*$ such that $z_2^{(i)} = r_{2(k - i) + 2}^{-1}$, for each $i = 1, \dots, k$.
    Similarly, since each $r_{2i - 1} \in \mathcal{B}_1(R)$, there exists $(y_1^{(i)}, y_2^{(i)}) \in \mathcal{A}(R)^*$ such that $y_2^{(i)} = r_{2(k - i) + 1}$, for each $i = 1, \dots, k$.
    Define
    \[
        x_1^{(i)} := \frac{y_1^{(i)}}{y_2^{(i)}} - \frac{z_1^{(i)}}{z_2^{(i)}} \quad \text{and} \quad
        x_2^{(i)} := \left( \frac{1}{y_2^{(i)}} - \frac{1}{z_2^{(i)}} \right) 
                 - \frac{z_1^{(i)}}{z_2^{(i)}} \left( \frac{\bar{y}_1^{(i)}}{\bar{y}_2^{(i)}} - \frac{\bar{z}_1^{(i)}}{\bar{z}_2^{(i)}} \right).
    \]
    Then $(x_1^{(i)}, x_2^{(i)}) \in \mathcal{A}(R)$ for each $i = 1, \dots, k$.

    Now define $C^{(0)} := h(r)$. For each $i \in {1, \dots, k}$, inductively define
    \begin{align*}
        A^{(i)} &:= C^{(i-1)} \cdot x_{+}(z_1^{(i)}, z_2^{(i)}), \\
        B^{(i)} &:= A^{(i)} \cdot x_{-}(x_1^{(i)}, x_2^{(i)}), \\
        C^{(i)} &:= B^{(i)} \cdot x_{+}(y_1^{(i)}, y_2^{(i)})^{-1}.
    \end{align*}

    For each $i = 1, \dots, k$, define
    \[
        s_i := r_1 r_2 \cdots r_{2(k - i) + 1} r_{2(k - i) + 2}, \quad \text{and} \quad s_{k+1} := 1.
    \]
    A straightforward computation shows that for each $i = 1, \dots, k$, the matrices $A^{(i)}$, $B^{(i)}$, and $C^{(i)}$ take the form:
    \[
        A^{(i)} = \begin{pmatrix}
            s_i & s_i z_1^{(i)} & s_i z_2^{(i)} \\
            * & * & * \\
            * & * & *
        \end{pmatrix}, \quad
        B^{(i)} = \begin{pmatrix}
            s_{i+1} & s_{i+1} y_1^{(i)} & s_{i+1} y_2^{(i)} \\
            * & * & * \\
            * & * & *
        \end{pmatrix}, \quad
        C^{(i)} = \begin{pmatrix}
            s_{i+1} & 0 & 0 \\
            * & * & * \\
            * & * & *
        \end{pmatrix}.
    \]

    Observe that
    \[
        C^{(k)} = \begin{pmatrix}
            1 & 0 & 0 \\
            c_{21} & c_{22} & c_{23} \\
            c_{31} & c_{32} & c_{33}
        \end{pmatrix}.
    \]
    By condition (\ref{U'23}), we have $c_{23} = 0$; by (\ref{U'11}), it follows that $c_{22} = 1$; and from (\ref{U'33}), we obtain $c_{33} = 1$. Therefore, $C^{(k)} \in U_{\sigma}^{-}(3, R)$.

    Moreover, by construction,
    \[
        C^{(k)} \in C^{(k-1)} \cdot U^{+}_{\sigma}(3, R) \cdot U^{-}_{\sigma}(3, R) \cdot U^{+}_{\sigma}(3, R).
    \]
    Hence, by iterating this relation,
    \[
        C^{(k)} \in C^{(0)} \cdot \left(U^{+}_{\sigma}(3, R) \cdot U^{-}_{\sigma}(3, R) \cdot U^{+}_{\sigma}(3, R)\right)^k = C^{(0)} \cdot \left(U^{+}_{\sigma}(3, R) \cdot U^{-}_{\sigma}(3, R)\right)^k \cdot U^{+}_{\sigma}(3, R).
    \]
    Since $C^{(0)} = h(r)$ and $C^{(k)} \in U^{-}_{\sigma}(3, R)$, we conclude that
    \[
        h(r) \in \left(U^{-}_{\sigma}(3, R) \cdot U^{+}_{\sigma}(3, R)\right)^{k+1}
    \]
    for every $r \in R^*$. Conjugating both sides by $w(1)$, we then obtain
    \[
        h(r) \in \left(U^{+}_{\sigma}(3, R) \cdot U^{-}_{\sigma}(3, R)\right)^{k+1}
    \]
    for every $r \in R^*$. 
\end{proof}

\begin{cor}\label{cor:unitriangular for SU(3,R)}
    \normalfont
    Let $R$ be a $\theta$-complete ring with $\mathcal{C}$-length $k$. Suppose $R$ satisfies $(SSR_1)$ condition. Then 
    \begin{align*}
        SU(3, R) = E_\sigma(3, R) = E'_\sigma(3,R) 
        &= \left( U^{+}_{\sigma} (3, R) \, U^{-}_{\sigma} (3, R) \right)^{k+1} U^{+}_{\sigma} (3, R) \\
        &= \left( U^{-}_{\sigma} (3, R) \, U^{+}_{\sigma} (3, R) \right)^{k+1} U^{-}_{\sigma} (3, R).
    \end{align*}
\end{cor}

\begin{proof}
    Immediate from Theorem~\ref{thm:SSR_1} and Theorem~\ref{thm:T is as unitriangular}.
\end{proof}


Finally, we conclude this section by addressing the following question: under what conditions do $T_\sigma(\Phi, R)$ and $H'_\sigma(\Phi, R)$ coincide? This question has already been answered in the case where $\Phi_\rho \not\sim {}^2A_{2n}$; in such cases, the two subgroups coincide for any commutative ring $R$. However, when $\Phi_\rho \sim {}^2A_{2n} \ (n \geq 1)$, additional conditions are needed to ensure the subgroups coincide.


\begin{prop}\label{prop:T=H=H' for general case}
    \normalfont
    Let $\Phi_\rho$ be a twisted root system of type ${}^2 A_{n} \ (n \geq 2), {}^2 D_{n} \ (n \geq 4), {}^2 E_6$, or ${}^3 D_4$.
    Let $R$ be a commutative ring with unity, and suppose there exists an automorphism $\theta: R \to R$ of order $2$ if $\Phi_\rho \not\sim {}^3D_4$, and of order $3$ if $\Phi_\rho \sim {}^3D_4$. 
    Assume further that $R$ is $\theta$-complete in the case where $\Phi_\rho \sim {}^2A_{2n}$. 
    Then
    \[
        T_\sigma(\Phi, R) = H_\sigma(\Phi, R) = H'_\sigma(\Phi, R).
    \]
\end{prop}

\begin{proof}
    Recall that it is well known that $T(\Phi, R) = H(\Phi, R)$ for all commutative ring $R$ when $\Phi \sim A_n \ (n \geq 2), D_n \ (n \geq 4), E_6$. 
    It follows that $T_\sigma (\Phi, R) = H_\sigma (\Phi, R)$ for all $R$ whenever $\Phi_\rho \sim {}^2 A_{n} \ (n \geq 2), {}^2 D_{n} \ (n \geq 4), {}^2 E_6$ or ${}^3 D_4$. 
    Moreover, it is also known that $H_\sigma(\Phi, R) = H'_\sigma(\Phi, R)$ for all $R$ whenever $\Phi_\rho \not\sim {}^2 A_{2n}$. 
    Therefore, to prove the proposition, it suffices to consider the case $\Phi_\rho \sim {}^2 A_{2n}$ and show that
    \[
        H_\sigma(A_{2n}, R) \subseteq H'_\sigma(A_{2n}, R)
    \]
    under the assumption that $R$ is $\theta$-complete.
    
    Let $h \in H_\sigma(A_{2n}, R)$. Then $h$ can be expressed as a product
    \[
        h = \prod_{[\alpha] \in \Phi_\rho} h_{[\alpha]}(t_{[\alpha]}),
    \]
    where each $t_{[\alpha]} \in R_{[\alpha]}^*$.  
    If $[\alpha] \sim A_1^2$, then it is immediate that $h_{[\alpha]}(t_{[\alpha]}) \in H'_\sigma(A_{2n}, R)$.  
    If $[\alpha] \sim A_2$, then by Proposition~\ref{prop:T=H=H'}, we also have $h_{[\alpha]}(t_{[\alpha]}) \in H'_\sigma(A_{2n}, R)$.  
    Hence, $h \in H'_\sigma(A_{2n}, R)$, completing the proof. 
\end{proof}


\begin{cor}\label{cor:E=E' for general case}
    \normalfont
    Let $\Phi$, $R$ and $\theta$ be as in Proposition~\ref{prop:T=H=H' for general case}. In addition, assume that $R$ is semi-local ring. Then 
    \[
        G_\sigma (\Phi, R) = E_\sigma (\Phi, R) = E'_\sigma (\Phi, R).
    \]
\end{cor}

\begin{proof}
    This follows immediately from Proposition~\ref{prop:T=H=H' for general case} and Proposition~5.8 in \cite{SG&DM1}.
\end{proof}


\section{Triangular and Unitriangular Factorization}

We are now ready to prove Theorem~\ref{thm:triangular decomposition} and Theorem~\ref{thm:unitriangular decomposition}. The main tool is a rank reduction theorem due to O. Tavgen~\cite{OT1} (cf. Theorem~\ref{thm:rank reduction theorem} of the present paper). In the original paper~\cite{OT1}, this theorem is stated for all (twisted) root systems excluding the case of ${}^2 A_{2n}$. However, as observed by A. Smolensky~\cite{AS}, the statement remains valid even for ${}^2 A_{2n}$. For the convenience of the reader, we include a full proof here.

Let $\Phi_\rho$ be a twisted root system of type ${}^2 A_{n} \ (n \geq 2), {}^2 D_n \ (n \geq 4), {}^2 E_6$, or ${}^3 D_4$. Fix a simple system $\Delta_\rho = \{ [\alpha_1], \dots, [\alpha_\ell] \}$ of $\Phi_\rho$. We rearrange the simple roots such that $[\alpha_1], \dots, [\alpha_\ell]$ represent the positions of the roots in the Dynkin diagram from the leftmost to the rightmost end.
For any $[\alpha] \in \Phi_\rho$, write its expansion in terms of the simple roots as
\[
    [\alpha] = \sum_{k=1}^\ell m_k([\alpha]) [\alpha_k],
\]
where $m_k([\alpha])$ denotes the coefficient of $[\alpha_k]$.

For each $i \in \{ 1, \dots, \ell \}$, define the subset $S_i = \{ [\alpha] \in \Phi_\rho \mid m_i([\alpha]) \geq 0 \}$. Then $S_i$ is a closed subset of $\Phi_\rho$ (cf. Section~\ref{subsec:parabolic}).
Let
\[
    \Phi_i := S_i^r = \{ [\alpha] \in \Phi_\rho \mid m_i ([\alpha]) = 0 \}
    \quad \text{and} \quad
    \Sigma_i := S_i^u = \{ [\alpha] \in \Phi_\rho \mid m_i ([\alpha]) > 0 \}.
\]
By Levi decomposition, we have
\[
    E'_\sigma (S_i, R) = E'_\sigma (\Phi_i, R) \ltimes E'_\sigma (\Sigma_i, R).
\]
Note that $\Phi_i$ is a root subsystem of $\Phi_\rho$. Let $\Phi_i^{\pm} = \Phi_i \cap \Phi^{\pm}$, and define
\[
    U_\sigma^{\pm} (\Phi_i, R) = E'_\sigma (\Phi_i^{\pm}, R) 
    \quad \text{and} \quad 
    U^{\pm}_\sigma (\Sigma_i, R) = E'_\sigma (\pm \Sigma_i, R).
\]


\begin{lemma}\label{lemma:Levi decomposition}
    \normalfont
    The group $\langle U_\sigma^{\pm}(\Phi_i, R), U_\sigma^{\pm}(\Sigma_i, R) \rangle$ is a semidirect product of $U_\sigma^{\pm}(\Phi_i, R)$ and $U_\sigma^{\pm}(\Sigma_i, R)$, with $U_\sigma^{\pm}(\Sigma_i, R)$ being a normal subgroup (this applies to all four possible combinations of signs). Moreover, 
    \[
        U_\sigma^{+} (\Phi, R) = U_\sigma^{+} (\Phi_i, R) \ltimes U_\sigma^{+} (\Sigma_i, R)
        \quad \text{and} \quad
        U_\sigma^{-} (\Phi, R) = U_\sigma^{-} (\Phi_i, R) \ltimes U_\sigma^{-} (\Sigma_i, R).
    \]
\end{lemma}

The proof of this lemma is well-known and follows directly from the Chevalley commutator formulas; we omit the details. 
Before we state the rank reduction theorem, we note the following auxiliary lemma, which is also well known.

\begin{lemma}\label{lemma:symmetric generator}
    Let $G$ be a group, and let $X \subset G$ be a symmetric generating set (i.e., $X^{-1} = X$). Suppose $Y \subset G$ is non-empty and satisfies $XY \subset Y$. Then $Y = G$.
\end{lemma}

\begin{proof}
    Since $Y$ is non-empty, choose an element $y_0 \in Y$. We aim to show that every element $g \in G$ lies in $Y$.  
    Consider $g' = g y_0^{-1} \in G$. Since $X$ generates $G$, there exist elements $x_1, \dotsc, x_n \in X$ such that $g' = x_1 \cdots x_n$.
    Thus, $g = g' y_0 = x_1 \cdots x_n y_0$.
    Since $XY \subset Y$ and $X$ is symmetric, it follows that multiplying elements of $Y$ on the left by elements of $X$ keeps the result in $Y$. Applying this inductively, we conclude that $g \in Y$. Therefore, $Y = G$.
\end{proof}


\begin{thm}\label{thm:rank reduction theorem}
    \normalfont
    Let $\Phi_\rho$ be a twisted root system of type ${}^2 A_{n} \ (n \geq 2), {}^2 D_{n} \ (n \geq 4), {}^2 E_6$, or ${}^3 D_4$.
    Let $R$ be a commutative ring with unity, and suppose there exists an automorphism $\theta: R \to R$ of order $2$ if $\Phi_\rho \not\sim {}^3D_4$, and of order $3$ if $\Phi_\rho \sim {}^3D_4$. 
    \begin{enumerate}[(a)]
        \item Assume that for the subsystems $\Phi_i$ $(i = 1 \text{ and } \ell)$, the elementary twisted Chevalley group $E'_\sigma(\Phi_i, R)$ admits a triangular factorization of the form
        \[
            E'_\sigma(\Phi_i, R) = H'_\sigma(\Phi_i, R) \, \underbrace{U_\sigma^{+}(\Phi_i, R) \, U_\sigma^{-}(\Phi_i, R) \dotsm U_\sigma^{\pm}(\Phi_i, R)}_{L_a \text{ factors}}
        \]
        of length $L_a$ for some $L_a \in \mathbb{N}$. Then the elementary twisted Chevalley group $E'_\sigma(\Phi, R)$ also admits a triangular factorization
        \[
            E'_\sigma(\Phi, R) = H'_\sigma(\Phi, R) \, \underbrace{U_\sigma^{+}(\Phi, R) \, U_\sigma^{-}(\Phi, R) \dotsm U_\sigma^{\pm}(\Phi, R)}_{L_a \text{ factors}}
        \]
        of the same length $L_a$.
    
        \item Assume that for the subsystems $\Phi_i$ $(i = 1 \text{ and } \ell)$, the elementary twisted Chevalley group $E'_\sigma(\Phi_i, R)$ admits a unitriangular factorization of the form
        \[
            E'_\sigma(\Phi_i, R) = \underbrace{U_\sigma^{+}(\Phi_i, R) \, U_\sigma^{-}(\Phi_i, R) \dotsm U_\sigma^{\pm}(\Phi_i, R)}_{L_b \text{ factors}}
        \]
        of length $L_b$ for some $L_b \in \mathbb{N}$. Then the elementary twisted Chevalley group $E'_\sigma(\Phi, R)$ also admits a unitriangular factorization
        \[
            E'_\sigma(\Phi, R) = \underbrace{U_\sigma^{+}(\Phi, R) \, U_\sigma^{-}(\Phi, R) \dotsm U_\sigma^{\pm}(\Phi, R)}_{L_b \text{ factors}}
        \]
        of the same length $L_b$.
    \end{enumerate}
\end{thm}

\begin{proof}
    First, observe that the set
    \[
        X = \{ x_{[\alpha]}(t) \mid \pm [\alpha] \in \Delta_\rho, \, t \in R_{[\alpha]} \}
    \]
    is a symmetric generating set for the group $E'_\sigma(\Phi, R)$. 
    
    \medskip

    \noindent (a) Define the subsets
    \[
        Y_a = H'_\sigma(\Phi, R) \, \underbrace{U_\sigma^{+}(\Phi, R) \, U_\sigma^{-}(\Phi, R) \dotsm U_\sigma^{\pm}(\Phi, R)}_{L_a \text{ factors}} \subset E'_\sigma(\Phi, R).
    \]
    To prove part~(a), we need to show that $Y_a = E'_\sigma (\Phi, R)$.
    By Lemma~\ref{lemma:symmetric generator}, it suffices to prove that $X Y_a \subset Y_a$.

    Let $x_{[\alpha]}(t) \in X$. Since $\mathrm{rk}(\Phi) \geq 2$, the root $[\alpha]$ belongs to at least one of the subsystems $\Phi_i$ for $i = 1$ or $\ell$. Choose such an index $i$, and let $\Sigma_i$ be defined as above. Then, by Lemma~\ref{lemma:Levi decomposition}, we can write 
    \[
        U_\sigma^{\pm}(\Phi, R) = U_\sigma^{\pm}(\Phi_i, R) \, U_\sigma^{\pm}(\Sigma_i, R).
    \]
    Moreover, $U_{\sigma}^{\pm}(\Phi_i, R)$ normalizes $U_{\sigma}^{\pm}(\Sigma_i, R)$ (for all four possible sign combinations).
    Therefore, we have
    \[
        Y_a = H'_\sigma(\Phi, R) \, U_\sigma^{+}(\Phi_i, R) \, U_\sigma^{-}(\Phi_i, R) \dotsm U_\sigma^{\pm}(\Phi_i, R) \cdot U_\sigma^{+}(\Sigma_i, R) \dotsm U_\sigma^{\pm}(\Sigma_i, R).
    \]
    Since $[\alpha] \in \Phi_i$, it follows that
    \[
        x_{[\alpha]}(a) \, E'_\sigma(\Phi_i, R) \subset E'_\sigma(\Phi_i, R),
    \]
    for every $a \in R_{[\alpha]}$. By our assumption of part (a),
    \[
        E'_\sigma(\Phi_i, R) = H'_\sigma(\Phi_i, R) \, U_\sigma^{+}(\Phi_i, R) \, U_\sigma^{-}(\Phi_i, R) \dotsm U_\sigma^{\pm}(\Phi_i, R).
    \]
    Hence,
    \begin{align*}
        & \hspace{-10mm} x_{[\alpha]}(t) \, H'_\sigma(\Phi, R) \, U_\sigma^{+}(\Phi_i, R) \, U_\sigma^{-}(\Phi_i, R) \dotsm U_\sigma^{\pm}(\Phi_i, R) \\
        &= x_{[\alpha]}(t) \, H'_\sigma(\Phi, R) \, E'_\sigma(\Phi_i, R) \\
        &= H'_\sigma(\Phi, R) \, x_{[\alpha]}(t') \, E'_\sigma(\Phi_i, R) \\
        &\subset H'_\sigma(\Phi, R) \, E'_\sigma(\Phi_i, R) \\
        &= H'_\sigma(\Phi, R) \, U_\sigma^{+}(\Phi_i, R) \, U_\sigma^{-}(\Phi_i, R) \dotsm U_\sigma^{\pm}(\Phi_i, R).
    \end{align*}
    Therefore, $x_{[\alpha]}(t) Y_a \subset Y_a$.
    
    \medskip
    
    \noindent (b) Similarly, define the subsets 
    \[
        Y_b = \underbrace{U_\sigma^{+}(\Phi, R) \, U_\sigma^{-}(\Phi, R) \dotsm U_\sigma^{\pm}(\Phi, R)}_{L_b \text{ factors}} \subset E'_\sigma(\Phi, R).
    \]
    To prove part~$(b)$, we need to show that $Y_b = E'_\sigma (\Phi, R)$.
    By Lemma~\ref{lemma:symmetric generator}, it suffices to prove that $X Y_b \subset Y_b$.

    Let $x_{[\alpha]}(t) \in X$. Since $\mathrm{rk}(\Phi) \geq 2$, the root $[\alpha]$ belongs to at least one of the subsystems $\Phi_i$ for $i = 1$ or $\ell$. Choose such an index $i$, and let $\Sigma_i$ be defined as above. Then, by Lemma~\ref{lemma:Levi decomposition}, we can write 
    \[
        U_\sigma^{\pm}(\Phi, R) = U_\sigma^{\pm}(\Phi_i, R) \, U_\sigma^{\pm}(\Sigma_i, R).
    \]
    Moreover, $U_{\sigma}^{\pm}(\Phi_i, R)$ normalizes $U_{\sigma}^{\pm}(\Sigma_i, R)$ (for all four possible sign combinations).
    Therefore, we have
    \[
        Y_b = U_\sigma^{+}(\Phi_i, R) \, U_\sigma^{-}(\Phi_i, R) \dotsm U_\sigma^{\pm}(\Phi_i, R) \cdot U_\sigma^{+}(\Sigma_i, R) \, U_\sigma^{-}(\Sigma_i, R) \dotsm U_\sigma^{\pm}(\Sigma_i, R).
    \]
    Again, since $[\alpha] \in \Phi_i$, we have
    \[
        x_{[\alpha]}(t) \, E'_\sigma(\Phi_i, R) \subset E'_\sigma(\Phi_i, R),
    \]
    for every $a \in R_{[\alpha]}$. By the assumption in part (b),
    \[
        E'_\sigma(\Phi_i, R) = U_\sigma^{+}(\Phi_i, R) \, U_\sigma^{-}(\Phi_i, R) \dotsm U_\sigma^{\pm}(\Phi_i, R).
    \]
    Hence,
    \[
        x_{[\alpha]}(t) \, U_\sigma^{+}(\Phi_i, R) \, U_\sigma^{-}(\Phi_i, R) \dotsm U_\sigma^{\pm}(\Phi_i, R)
        \subset U_\sigma^{+}(\Phi_i, R) \, U_\sigma^{-}(\Phi_i, R) \dotsm U_\sigma^{\pm}(\Phi_i, R).
    \]
    Therefore, $x_{[\alpha]}(t) Y_b \subset Y_b$.
\end{proof}


To complete the proofs of Theorem~\ref{thm:triangular decomposition} and Theorem~\ref{thm:unitriangular decomposition}, it remains to verify the corresponding base cases. These are established by Lemma~\ref{lemma:base case for triangular} and Lemma~\ref{lemma:base case for unitriangular}, respectively. 

\begin{lemma}\label{lemma:base case for triangular}
    \normalfont
    Let $\Phi_\rho$ be a twisted root system of type ${}^2 A_n \ (n \geq 2), {}^2 D_n \ (n \geq 4), {}^2E_6$ or ${}^3 D_4$.
    Let $R$ be a commutative ring with unity, and suppose there exists a ring automorphism $\theta: R \longrightarrow R$ of order $2$ if $\Phi_\rho \not\sim {}^3 D_4$, or of order $3$ if $\Phi_\rho \sim {}^3 D_4$. 
    Fix a class $[\alpha] \in \Phi_\rho$, and let $\Phi_{[\alpha]}$ denote the root subsystem $\{ \pm[\alpha] \}$. Assume the following conditions:
    \begin{enumerate}[(a)]
        \item $R_\theta$ satisfies the $(SR_1)$ condition if $[\alpha] \sim A_1$;
        \item $R$ satisfies the $(SR_1)$ condition if $[\alpha] \sim A_1^2$ or $A_1^3$;
        \item $R$ satisfies the $(SSR_1)$ condition if $[\alpha] \sim A_2$.
    \end{enumerate}
    Then
    \[
        E'_\sigma(\Phi_{[\alpha]}, R) = H'_\sigma(\Phi_{[\alpha]}, R) \, U^{+}_\sigma(\Phi_{[\alpha]}, R) \, U^{-}_\sigma(\Phi_{[\alpha]}, R) \, U^{+}_\sigma(\Phi_{[\alpha]}, R).
    \]
\end{lemma}

\begin{proof}
    If $[\alpha] \sim A_1$, $A_1^2$, or $A_1^3$, the result follows from Lemma~2.1 of \cite{SSV2}. If $[\alpha] \sim A_2$, it follows from Theorem~\ref{thm:SSR_1}.
\end{proof}

\begin{lemma}\label{lemma:base case for unitriangular}
    \normalfont
    Let $\Phi_\rho$ be a twisted root system of type ${}^2 A_n \ (n \geq 2), {}^2 D_n \ (n \geq 4), {}^2E_6$ or ${}^3 D_4$.
    Let $R$ be a commutative ring with unity, and suppose there exists a ring automorphism $\theta: R \longrightarrow R$ of order $2$ if $\Phi_\rho \not\sim {}^3 D_4$, or of order $3$ if $\Phi_\rho \sim {}^3 D_4$. 
    Fix a class $[\alpha] \in \Phi_\rho$, and let $\Phi_{[\alpha]}$ denote the root subsystem $\{ \pm[\alpha] \}$. Assume the following conditions:
    \begin{enumerate}[(a)]
        \item $R_\theta$ satisfies the $(SR_1)$ condition if $[\alpha] \sim A_1$;
        \item $R$ satisfies the $(SR_1)$ condition if $[\alpha] \sim A_1^2$ or $A_1^3$;
        \item $R$ is $\theta$-complete with $\mathcal{C}$-length $k$ and satisfies the $(SSR_1)$ condition if $[\alpha] \sim A_2$.
    \end{enumerate}
    Then
    \[
        E'_\sigma(\Phi_{[\alpha]}, R) = 
        \begin{cases}
            \left( U^{+}_\sigma(\Phi_{[\alpha]}, R) \, U^{-}_\sigma(\Phi_{[\alpha]}, R) \right)^2 & \text{if } [\alpha] \sim A_1, A_1^2, \text{ or } A_1^3, \\
            \left( U^{+}_\sigma(\Phi_{[\alpha]}, R) \, U^{-}_\sigma(\Phi_{[\alpha]}, R) \right)^{k+1} \, U^{+}_\sigma(\Phi_{[\alpha]}, R) & \text{if } [\alpha] \sim A_2.
        \end{cases}
    \]
\end{lemma}

\begin{proof}
    If $[\alpha] \sim A_1$, $A_1^2$, or $A_1^3$, the result follows from Lemma~1 of \cite{SSV1}. If $[\alpha] \sim A_2$, it follows from Corollary~\ref{cor:unitriangular for SU(3,R)}.
\end{proof}




\begin{thebibliography}{AAAA}

\bibitem{EA1} Eiichi Abe, \emph{Coverings of twisted Chevalley groups over commutative rings}, Sci. Rep. Tokyo Kyoiku Daigaku Sect. A, Vol. 13 (1977), 194--218.

\bibitem{EA2} Eiichi Abe, \emph{Chevalley groups over local rings}, Tohoku Math. J. (2), Vol. 21 (1969), 474--494.

\bibitem{EA3} Eiichi Abe, \emph{Chevalley groups over commutative rings}, Radical theory, Proc. 1988, Sendai Conf., Vol. 83 (1989), 1--23.

\bibitem{EA&KS} Eiichi Abe and Kazuo Suzuki, \emph{On normal subgroups of Chevalley groups over commutative rings}, Tohoku Math. J. (2), Vol. 28 (1976), 185--198.

\bibitem{LB,NN&LP} L. Babai, N. Nikolov and L. Pyber, \emph{Product growth and mixing in finite groups}, Proceedings of the {N}ineteenth {A}nnual {ACM}-{SIAM} {S}ymposium on {D}iscrete {A}lgorithms (2008), 248--257.

\bibitem{HB} Hyman Bass, \emph{$K$-theory and stable algebra}, Institut des Hautes \'Etudes Scientifiques. Publications Math\'ematiques, No. 22 (1964), 5--60.

\bibitem{DC&GK1} David Carter and Gordon Keller, \emph{Bounded elementary generation of {$\mathrm{SL}_{n}({\mathcal{O}})$}}, American Journal of Mathematics, Vol. 105 (1983), 673--687.

\bibitem{DC&GK2} David Carter and Gordon Keller, \emph{Elementary expressions for unimodular matrices}, Communications in Algebra, Vol. 12 (1984), 379--389.

\bibitem{RC} R. W. Carter, \emph{Simple groups of Lie type}, 2nd edition, Wiley, London (1989).

\bibitem{HC&MC} Huanyin Chen and Miaosen Chen, \emph{On products of three triangular matrices over associative rings}, Linear Algebra and its Applications, Vol. 387 (2004), 297--311.

\bibitem{SG&DM1} Shripad M. Garge and Deep H. Makadiya, \emph{On normal subgroups of twisted Chevalley groups over commutative rings}, \hyperlink{arXiv:2502.04766}{arXiv:2502.04766} (2025).

\bibitem{JH} James E. Humphreys, \emph{Introduction to Lie algebras and representation theory}, Graduate texts in Mathematics, Springer-Verlag, New York-Berlin (1972).

\bibitem{DM} Dave W. Morris, \emph{Bounded generation of {${\rm SL}(n,A)$} (after {D}. {C}arter, {G}. {K}eller, and {E}. {P}aige)}, New York Journal of Mathematics, Vol. 13 (2007), 383--421.

\bibitem{KN,MD&TS} K. R. Nagarajan, M. P. Devasahayam and T. Soundararajan, \emph{Products of three triangular matrices over commutative rings}, Linear Algebra and its Applications, Vol. 348 (2002), 1--6.

\bibitem{RS} Robert Steinberg, \emph{Lectures on Chevalley Groups}, Yale University Press (1968).

\bibitem{RSTCG} Robert Steinberg, \emph{Variations on a Theme of Chevalley}, Pacific J. Math., Vol. 9 (1959), 875-891.

\bibitem{SSV1} A. V. Smolensky, B. Sury and N. A. Vavilov, \emph{Unitriangular factorizations of Chevalley groups}, Zapiski Nauchn. Semin. POMI, Vol. 388 (2011), 17--47.

\bibitem{SSV2} A. V. Smolensky, B. Sury and N. A. Vavilov, \emph{Gauss decomposition for Chevalley groups, revisited}, International Journal of Group Theory, Vol. 1 (2012), 3--16.

\bibitem{AS} A. V. Smolensky, \emph{Unitriangular factorizations of twisted Chevalley groups}, International Journal of Algebra and Computation, Vol. 23 (2013), 1497--1502.

\bibitem{KS1} Kazuo Suzuki, \emph{On Normal Subgroups of Twisted Chevalley Groups over Local Rings}, Sci. Rep. Tokyo Kyoiku Daigaku Sect. A, Vol. 13 (1977), 238--249.

\bibitem{KS2} Kazuo Suzuki, \emph{Normality of the Elementary Subgroups of Twisted Chevalley Groups over Commutative Rings}, J. Algebra, Vol. 175 (1995), 526--536.

\bibitem{KS3} Kazuo Suzuki, \emph{Centers of Twisted Chevalley Groups over Commutative Rings}, Kumamoto J. Math., Vol. 6 (1993), 1--9.

\bibitem{GT} Giovanni Taddei, \emph{Normalit\'{e} des groupes \'{e}l\'{e}mentaires dans les groupes de {C}hevalley sur un anneau}, Contemp. Math. (2), Vol. 55 (1986), 693--710.

\bibitem{OT1} O. I. Tavgen, \emph{Bounded generation of Chevalley groups over rings of {$S$}-integer algebraic numbers}, Izvestiya Akademii Nauk SSSR. Seriya Matematicheskaya, Vol. 54 (1) (1990), 97--122. 

\bibitem{OT2} O. I. Tavgen, \emph{Bounded generation of normal and twisted Chevalley groups over the rings of {$S$}-integers}, Proceedings of the {I}nternational {C}onference on {A}lgebra, {P}art 1 ({N}ovosibirsk, 1989), Vol. 131, Part 1 (1992), 409--421.

\end{thebibliography}
\end{document}